\documentclass[twoside,11pt]{article}

\usepackage{graphicx, subfigure}
\usepackage[top=1in,bottom=1in,left=1in,right=1in]{geometry}
\usepackage[sort&compress, numbers]{natbib} 
\usepackage{amsfonts, amsmath, amssymb, amsthm, constants, bbm}
\usepackage{MnSymbol}
\usepackage{mathtools}
\mathtoolsset{showonlyrefs}
\usepackage{enumitem}

\usepackage{color}
\definecolor{darkred}{RGB}{100,0,0}
\definecolor{darkgreen}{RGB}{0,100,0}
\definecolor{darkblue}{RGB}{0,0,150}

\usepackage{hyperref}
\hypersetup{colorlinks=true, linkcolor=darkred, citecolor=darkgreen, urlcolor=darkblue}
\usepackage{url}

\def\I{{\rm I}}

\def\hcrank{{\rm HC}_{\rm rank}}

\newtheorem{thm}{Theorem}
\newtheorem{prp}{Proposition}
\newtheorem{lem}{Lemma}

\theoremstyle{remark}

\newtheorem{rem}{Remark}

\def\beq{\begin{equation}} 
\def\eeq{\end{equation}}
\def\beqn{\begin{eqnarray*}}
\def\eeqn{\end{eqnarray*}}
\def\Bitem{\begin{itemize}\setlength{\itemsep}{.2in}}
\def\bitem{\begin{itemize}\setlength{\itemsep}{.05in}}
\def\eitem{\end{itemize}}
\def\Benum{\begin{enumerate}\setlength{\itemsep}{.2in}}
\def\benum{\begin{enumerate}\setlength{\itemsep}{.05in}}
\def\eenum{\end{enumerate}}
\def\bmult{\begin{multline*}}
\def\emult{\end{multline*}}
\def\bcenter{\begin{center}}
\def\ecenter{\end{center}}
\def\bframe{\begin{frame}}
\def\eframe{\end{frame}}

\newcommand{\thmref}[1]{Theorem~\ref{thm:#1}}
\newcommand{\prpref}[1]{Proposition~\ref{prp:#1}}

\newcommand{\lemref}[1]{Lemma~\ref{lem:#1}}
\newcommand{\secref}[1]{Section~\ref{sec:#1}}
\newcommand{\figref}[1]{Figure~\ref{fig:#1}}

\DeclareMathOperator{\diag}{diag}

\def\cH{\mathcal{H}}

\def\cN{\mathcal{N}}

\def\bbI{\mathbb{I}}

\def\bbP{\mathbb{P}}

\def\bbR{\mathbb{R}}

\newcommand{\E}{\operatorname{\mathbb{E}}}
\renewcommand{\P}{\operatorname{\mathbb{P}}}
\newcommand{\Var}{\operatorname{Var}}

\def\iid{\stackrel{\rm iid}{\sim}}
\def\Bin{\text{Bin}}
\def\Unif{\text{Unif}}

\def\weak{\rightharpoonup}

\def\eps{\varepsilon}

\def\iff{\ \Leftrightarrow \ }

\def\1{\mathbbm{1}}
\newcommand{\IND}[1]{\bbI\{ #1 \}}

\definecolor{purple}{rgb}{0.4,.1,.9}

\newcommand\blfootnote[1]{%
\begingroup
\renewcommand\thefootnote{}\footnote{#1}%
\addtocounter{footnote}{-1}%
\endgroup
}

\begin{document}
\thispagestyle{empty}

\title{Detection of Sparse Positive Dependence}
\author{Ery Arias-Castro \and Rong Huang \and Nicolas Verzelen}
\date{}
\maketitle

\begin{abstract}
In a bivariate setting, we consider the problem of detecting a sparse contamination or mixture component, where the effect manifests itself as a positive dependence between the variables, which are otherwise independent in the main component.  We first look at this problem in the context of a normal mixture model.  In essence, the situation reduces to a univariate setting where the effect is a decrease in variance.  In particular, a higher criticism test based on the pairwise differences is shown to achieve the detection boundary defined by the (oracle) likelihood ratio test.
We then turn to a Gaussian copula model where the marginal distributions are unknown.  Standard invariance considerations lead us to consider rank tests.  In fact, a higher criticism test based on the pairwise rank differences achieves the detection boundary in the normal mixture model, although not in the very sparse regime.  We do not know of any rank test that has any power in that regime.
\end{abstract}

\blfootnote{EAC and RH are with the Department of Mathematics, University of California, San Diego. NV is with INRAE, Montpellier SupAgro, MISTEA,
Univ.~Montpellier.}

\section{Introduction} \label{sec:intro}
The detection of rare effects has been an important problem for years in settings, and may be particularly relevant today, for example, with the search for personalized care in the health industry, where a small fraction of a population may respond particularly well, or particularly poorly, to some given treatment \cite{redekop2013faces}.

Following a theoretical investigation initiated in large part by \citet{ingster1997some} and broadened by \citet{Jin:2004fj}, we are interested in studying two-component mixture models, also known as contamination models, in various asymptotic regimes defined by how the small mixture weight converges to zero.
Most of the existing work in the setting of univariate data has focused on models where the contamination manifests itself as a shift in mean \cite{donoho2015higher, donoho2008higher, hall2010innovated, cai2014optimal, moscovich2016exact} with a few exceptions where the effect is a change in variance \cite{AriasCastro:2018wr}, or a change in both mean and variance \cite{Cai:2011cb}.

In the present paper, we are interested in bivariate data, instead, and more specifically in a situation where the effect felt in the dependence between the two variables being measured.
This setting has been recently considered in the literature in the context of assessing the reproducibility of studies.
For example, \citet{Li:2011gg} aimed to identify significant features from separate studies using an expectation-maximization (EM) algorithm. They applied a copula mixture model and assumed that changes in the mean and covariance matrix differentiate the contaminated component from the null component.
\citet{Zhao:9999ku} studied another model where variables from the contamination are stochastically larger marginally. In both models, the marginal distributions have some non-null effects.
Similar settings have been considered within a multiple testing framework \cite{Bogomolov:2018gt, zhao2015false}.

While existing work has focused on models motivated by questions of reproducibility, in the present work we come back to basics and directly address the problem of detecting a bivariate mixture with a component where the variables are independent and a component where the variables are positively dependent.

\subsection{Gaussian Mixture Model}
\citet{ingster1997some} and \citet{Jin:2004fj} started with a mixture of Gaussians, and we do the same, and in our setting, this means we consider the following mixture model
\beq \label{model1}
(X, Y) \sim (1-\eps)\, \cN(0, \I) + \eps\, \cN(0, \Sigma_\rho),
\quad \Sigma_\rho := \begin{pmatrix} 1 & \rho \\ \rho & 1 \end{pmatrix},
\eeq
where $\eps \in [0,1/2)$ is the contamination proportion and $0 \le \rho \le 1$ is the correlation between the two variables under contamination.
We consider the following hypothesis testing problem: based on $(X_1, Y_1), \dots, (X_n, Y_n)$ drawn iid from \eqref{model1}, decide
\beq\label{problem1}
\cH_0: \eps = 0 \quad \text{versus} \quad \cH_1: \eps > 0,\, \rho > 0.
\eeq

Note that under the null hypothesis, $(X, Y)$ is from the bivariate standard normal. Under the alternative, $X$ and $Y$ remain standard normal marginally. Following the literature on the detection of sparse mixtures \cite{ingster1997some, Jin:2004fj}, we are most interested in a situation, asymptotic as $n \to \infty$, where $\eps = \eps_n \to 0$, and the central question is how large $\rho = \rho_n$ needs to be in order to reliability distinguish these hypotheses.

The formulation \eqref{model1} suggests that the alternative hypothesis is composite, but if we assume that $(\eps, \rho)$ are known under the alternative, then the likelihood ratio test (LRT) is optimal by Neyman-Pearson lemma.
We start with characterizing the behavior of the LRT, which provides a benchmark. We then study some other testing procedures that do not require knowledge of the model parameters:\footnote{ Such procedures are said to be {\em adaptive}.}
\bitem
\item The {\em covariance test} rejects for large values of $\sum_i X_i Y_i$, and coincides with Rao's score test in the present context. This is the classical test for independence, specifically designed for the case where $\eps = 1$ and $\rho > 0$ under the alternative. We shall see that it is suboptimal in some regimes.
\item The {\em extremes test} rejects for small values of $\min_i |X_i-Y_i|$.  This test exploits the fact that, because $\rho$ is assumed positive, the variables in the contaminated component are closer to each other than in the null component.
\item The {\em higher criticism test} was suggested by John Tukey and deployed by \citet{Jin:2004fj} for the testing of sparse mixtures.  We propose a version of that test based on the pairwise differences, $U_i := (X_i - Y_i)/\sqrt{2}$. In detail, the test rejects for large values of
\beq \label{HC}
\sup_{u \ge 0} \frac{\sqrt{n}\, (\hat F(u) - \Psi(u))}{\sqrt{\Psi(u)(1-\Psi(u))}},
\eeq
where $\Psi(u) := 2\Phi(u) - 1$, with $\Phi$ denotes the standard normal distribution function, and $\hat F(u) := \frac1n \sum_{i=1}^n \IND{|U_i| \le u}$, the empirical distribution function of $|U_1|, \dots, |U_n|$.
\eitem

As is common practice in this line of work \cite{ingster1997some, Jin:2004fj}, under $\cH_1$ we set
\beq \label{eps}
\eps = n^{-\beta}, \quad \beta \in (0, 1) \text{ fixed}.
\eeq
The setting where $\beta \le 1/2$ is often called the dense regime and the setting where $\beta > 1/2$ is often called the sparse regime.
Our analysis reveals the following:
\renewcommand{\theenumi}{(\alph{enumi})}
\renewcommand{\labelenumi}{\theenumi}
\benum  \setlength{\itemsep}{0in}
\item {\em Dense regime.}
The dense regime is most interesting when $\rho \to 0$.  In that case, we find that the covariance test and the higher criticism test match the asymptotic performance of the likelihood ratio test to first-order, while the extremes test has no power.
\item {\em Sparse regime.}
The sparse regime is most interesting when $\rho \to 1$.  In that case, we find that the higher criticism test still performs as well as the likelihood ratio test to first order, while the covariance test is powerless, and the extremes test is suboptimal.
\eenum


\subsection{Gaussian Mixture Copula Model}
From a practical point of view, the assumption that both $X$ and $Y$ are normally distributed is quite stringent. Hence, we would like to know if there are nonparametric procedures that do not require such a condition but can still achieve the same performance as the likelihood ratio test.
In the univariate setting where the effect arises as a shift in mean, this was investigated in \cite{AriasCastro:2016is}.
In the bivariate setting, in a model for reproducibility, \citet{Zhao:9999ku} proposed a nonparametric test based on a weighted version of Hoeffding's test for independence.

Here, instead of model \eqref{model1}, we suppose $(X, Y)$  follows a Gaussian mixture copula model (GMCM) \cite{bilgrau2016gmcm}, meaning that there is a latent random vector $(Z^1, Z^2)$ such that
\beq \label{model3}
F(X) = \Phi(Z^1), \quad G(Y) = \Phi(Z^2), \quad
(Z^1, Z^2) \sim (1-\eps) \cN(0, I) + \eps \cN(0, \Sigma_\rho),
\quad \Sigma_\rho := \begin{pmatrix} 1 & \rho \\ \rho & 1 \end{pmatrix},
\eeq
where $F$ and $G$ are unknown distribution functions on the real line, and $\Phi$ is the standard normal distribution function, while $\eps \in [0,1/2)$ is the contamination proportion and $0 \le \rho \le 1$ is the correlation between $Z^1$ and $Z^2$ in the contaminated component, as before in model \eqref{model1}.
\citet{Li:2011gg} also used a copula mixture model, but they placed emphasis on the mean while we focus on the dependence.

We still consider the testing problem \eqref{problem1}, but now in the context of Model \eqref{model3}. The setting is nonparametric in that both $F$ and $G$ are unknown.
Model \eqref{model3} is crafted in such a way that the marginal distributions of $X$ and $Y$ contain absolutely no information that is pertinent to the testing problem under consideration.

%

The model is also attractive because of an invariance under all increasing marginal transformations of the variables.  This is the same invariance that leads to considering rank based methods such as the Spearman correlation test \cite[Chp 6]{lehmann2006testing}.
In fact, we analyze the Spearman correlation test, which is the nonparametric analog to the covariance test, showing that it is first-order asymptotically optimal in the dense regime. We also propose and analyze a nonparametric version of the higher criticism based on ranks which we show is first-order asymptotically optimal in the moderately sparse regime where $1/2<\beta<3/4$.
In the very sparse regime, where $\beta > 3/4$, we do not know of any rank-based test that has any power.


\section{Gaussian Mixture Model} \label{sec:GMM}
In this section, we focus on the Gaussian mixture model \eqref{model1}.  We start by deriving a lower bound on the performance of the likelihood ratio test, which provides a benchmark for the other (adaptive) tests, which we subsequently analyze.

We distinguish between the dense and sparse regimes:
\begin{align}
\text{dense regime} & \quad \rho = n^{-\gamma}, \quad \gamma > 0 \text{ fixed}; \label{dense} \\
\text{sparse regime} & \quad \rho = 1 - n^{-\gamma}, \quad \gamma > 0 \text{ fixed}. \label{sparse}
\end{align}

We say that a testing procedure is asymptotically powerful (resp.~powerless) if the sum of its probabilities of Type I and Type II errors (its risk) has limit 0 (resp.~limit inferior at least 1) in the large sample asymptote.

\subsection{The likelihood ratio test}

\begin{thm}
Consider the testing problem \eqref{problem1} with $\eps$ parameterized as in \eqref{eps}.
In the dense regime, with $\rho$ parameterized as in \eqref{dense}, the likelihood ratio test is asymptotically powerless when $\gamma > 1/2 - \beta$.
In the sparse regime, with $\rho$ parameterized as in \eqref{sparse}, the likelihood ratio test is asymptotically powerless when $\gamma < 4(\beta - 1/2)$.
\end{thm}

This only provides a lower bound on what can be achieved, but it will turn out that to be sharp once we establish the performance of the higher criticism test in \prpref{hc} below.

\begin{proof}
The proof techniques are standard and already present in \cite{donoho2015higher,ingster1997some}, and many of the subsequent works.

Defining $U := (X-Y)/\sqrt{2}$ and $V := (X+Y)/\sqrt{2}$, the model \eqref{model1} is equivalently expressed in terms of $(U, V)$, which has distribution
\beq \label{model2}
(U, V) \sim (1-\eps) \cN(0, \I) + \eps \cN(0, \Delta_\rho), \quad \Delta_\rho := \diag(1-\rho, 1+\rho).
\eeq
Note that $U$ and $V$ are independent only conditional on knowing what distribution they were sampled from.
In terms of the $(U,V)$'s, the likelihood ratio is
\beq
L := \prod_{i=1}^n L_i,
\eeq
where $L_i$ is the likelihood ratio for observation $(U_i,V_i)$, which in the present case takes the following expression
\begin{align}
L_i
&= \frac{\frac{1 - \eps}{2\pi} \exp(-\frac12 U_i^2 - \frac12 V_i^2) + \frac\eps{2\pi \sqrt{1-\rho^2}} \exp(-\frac1{2(1-\rho)} U_i^2 -\frac1{2(1+\rho)} V_i^2)}{\frac1{2\pi} \exp(-\frac12 U_i^2 - \frac12 V_i^2)} \\
&= 1 -\eps + \eps (1-\rho^2)^{-1/2} \exp(-\tfrac\rho{2(1-\rho)} U_i^2 +\tfrac\rho{2(1+\rho)} V_i^2).
\end{align}
The risk of the likelihood ratio test is equal to \cite[Problem 3.10]{lehmann2006testing}
\beq
{\rm risk}(L) := 1 - \frac12 \E_0[ |L - 1| ].
\eeq
We show that ${\rm risk}(L) = 1 + o(1)$ under each of the stated conditions.  We consider each regime in turn.

\medskip\noindent
{\em Dense regime.}
It turns out that it suffices to bound the second moment.  Indeed, using the Cauchy-Schwarz inequality, we have
\beq
{\rm risk}(L) \ge 1 - \frac12 \sqrt{\E_0[L^2] - 1},
\eeq
reducing the task to showing that $\E_0[L^2] \le 1 + o(1)$.
We have
\beq
\E_0[L^2]
= \prod_{i=1}^n \E_0[L_i^2]
= (\E_0[L_1^2])^n
\eeq
where
\begin{align}
\E_0[L_1^2]
&= \E_0\Big[\Big(1 -\eps + \eps (1-\rho^2)^{-1/2} \exp(-\tfrac\rho{2(1-\rho)} U_1^2 +\tfrac\rho{2(1+\rho)} V_1^2)\Big)^2\Big] \\
&= (1 -\eps)^2 + 2 (1-\eps) \eps
+ \eps^2 (1-\rho^2)^{-1} \E_0 \big [\exp(-\tfrac\rho{(1-\rho)} U_1^2) \big ] \E_0 \big [\exp(\tfrac\rho{(1+\rho)} V_1^2) \big ]\\
&= 1 - \eps^2 + \eps^2 (1-\rho^2)^{-1} \E_0 \big [\exp(-\tfrac\rho{(1-\rho)} U_1^2) \big ] \E_0 \big [\exp(\tfrac\rho{(1+\rho)} V_1^2) \big].
\end{align}
For the third term, we have
\beq
\E_0 \big [\exp(-\tfrac\rho{(1-\rho)} U_1^2) \big]
= \frac1{\sqrt{2\pi}} \int_{-\infty}^\infty e^{-\frac\rho{1-\rho} u^2 - \frac12 u^2} {\rm d}u
= \sqrt{\frac{1-\rho}{1+\rho}},
\eeq
and
\beq
\E_0 \big [\exp(\tfrac\rho{(1+\rho)} V_1^2) \big]
= \frac1{\sqrt{2\pi}} \int_{-\infty}^{\infty} e^{\frac\rho{1+\rho} v^2 - \frac12 v^2} {\rm d}v
=\sqrt{\frac{1+\rho}{1-\rho}}.
\eeq
Hence, we have
\beq
\E_0[L_1^2] = 1 + \eps^2 \rho^2 / (1-\rho^2),
\eeq
and, therefore,
\beq
\E_0[L^2]
= \big[1 + \eps^2 \rho^2 / (1-\rho^2) \big]^n
\le \exp\big[n \eps^2 \rho^2 / (1-\rho^2)\big],
\eeq
so that $\E_0[L^2] \le 1 + o(1)$ when
\beq \label{lb1}
n \eps^2 \rho^2 = o(1),
\eeq
since $\rho$ is assumed to be bounded away from 1.
Under the specified parameterization, this happens exactly when $\gamma > 1/2 - \beta$.

\medskip\noindent
{\em Sparse regime.}
It turns out that simply bounding the second moment, as we did above, does not suffice.  Instead, we truncate the likelihood and study the behavior of its first two moments.
Define the indicator variable $D_i = \IND{|V_i| \le \sqrt{2 \log n}}$ and the corresponding truncated likelihood ratio
\beq
\bar L = \prod_{i=1}^n \bar L_i, \quad \bar L_i := L_i D_i.
\eeq
Using the triangle inequality, the fact that $\bar L\leq L$, and the Cauchy-Schwarz inequality, we have the following upper bound:
\begin{align}
\E_0|L-1|
&\le \E_0|\bar L-1|+\E_0 (L-\bar L) \\
&\le \big[ \E_0[\bar L^2] - 1 + 2 (1 - \E_0[\bar L]) \big]^{1/2} + (1-\E_0[\bar L]) \ ,
\end{align}
so that ${\rm risk}(L) = 1 + o(1)$ when $\E_0[\bar L^2] \le 1 + o(1)$ and $\E_0[\bar L] \ge 1 - o(1)$.

For the first moment, we have
\beq
\E_0[\bar L]
= \prod_{i=1}^n \E_0[\bar L_i]
= (\E_0[\bar L_1])^n
\eeq
where, using the independence of $U_1$ and $V_1$, and taking the expectation with respect to $U_1$ first,
\begin{align}
\E_0[\bar L_1]
&= \E_0\Big[\Big(1 -\eps + \eps (1+\rho)^{-1/2} \exp(\tfrac\rho{2(1+\rho)} V_1^2)\Big) D_1\Big] \\
&= (1 -\eps) \Psi(\sqrt{2 \log n}) + \eps \Psi(\sqrt{2 \log n}/\sqrt{1+\rho}) \\
&= (1 -\eps) (1 - O(n^{-1}/\sqrt{\log n})) + \eps (1 - O(n^{-1/(1+\rho)}/\sqrt{\log n})) \\
&= 1 - o(1/n) - o(\eps n^{-1/(1+\rho)}),
\end{align}
where, for $t \ge 0$,
\beq \label{def:psi}
\Psi(t) = \bbP(|\cN(0,1)| \le t) = 2\Phi(t) - 1 = \int_{-t}^t \frac{e^{-s^2/2}}{\sqrt{2\pi}} {\rm d}s,
\eeq
and we used the fact that $1 - \Psi(t) \asymp e^{-t^2/2}/t$ when $t \to \infty$.
Since $\eps = n^{-\beta}$ with $\beta > 1/2$ in the sparse regime, for $\rho$ sufficiently close to 1, $\eps n^{-1/(1+\rho)} \le 1/n$, in which case $\E_0[\bar L_1] \ge 1 - o(1/n)$.
This yields
\beq
\E_0[\bar L] \ge (1 - o(1/n))^n = 1 - o(1).
\eeq

For the second moment, we have
\beq
\E_0[\bar L^2]
= \prod_{i=1}^n \E_0[\bar L_i^2]
= \E_0[\bar L_1^2]^n,
\eeq
where
\begin{align}
\E_0[\bar L_1^2]
&= \E_0\Big[\Big(1 -\eps + \eps (1-\rho^2)^{-1/2} \exp(-\tfrac\rho{2(1-\rho)} U_1^2 +\tfrac\rho{2(1+\rho)} V_1^2)\Big)^2 D_1\Big] \\
&= (1 -\eps)^2 \Psi(\sqrt{2 \log n}) + 2 (1-\eps) \eps \Psi(\sqrt{2 \log n}/\sqrt{1+\rho}) \\
&\quad + \eps^2 (1-\rho^2)^{-1} \E_0[\exp(-\tfrac\rho{(1-\rho)} U_1^2)] \E_0[\exp(\tfrac\rho{(1+\rho)} V_1^2) D_1].
\end{align}
The sum of first two terms is bounded from above by $(1-\eps)^2 + 2 (1-\eps) \eps = 1 - \eps^2$.
For the third term, we have
\beq
\E_0[\exp(-\tfrac\rho{(1-\rho)} U_1^2)]
= \frac1{\sqrt{2\pi}} \int_{-\infty}^\infty e^{-\frac\rho{1-\rho} u^2 - \frac12 u^2} {\rm d}u
= \sqrt{\frac{1-\rho}{1+\rho}},
\eeq
and
\beq
\E_0[\exp(\tfrac\rho{(1+\rho)} V_1^2) D_1]
= \frac1{\sqrt{2\pi}} \int_{-\sqrt{2\log n}}^{\sqrt{2\log n}} e^{\frac\rho{1+\rho} v^2 - \frac12 v^2} {\rm d}v
\le \frac1{\sqrt{2\pi}} 2 \sqrt{2\log n},
\eeq
using the fact that $\rho \le 1$.
Hence,
\begin{align}
\E_0[\bar L_1^2]
&\le 1 -\eps^2 + \eps^2 (1-\rho^2)^{-1} \sqrt{\frac{1-\rho}{1+\rho}} \frac1{\sqrt{2\pi}} 2 \sqrt{2\log n} \\
&\le 1 + \eps^2 (1-\rho)^{-1/2} (\log n)^{1/2},
\end{align}
when $\rho$ is sufficiently close to 1.
This in turn yields the following bound
\beq
\E_0[\bar L^2]
\le \big[1 + \eps^2 (1-\rho)^{-1/2} (\log n)^{1/2}\big]^n
\le \exp\big[n \eps^2 (1-\rho)^{-1/2} (\log n)^{1/2}\big],
\eeq
so that $\E_0[\bar L^2] \le 1 + o(1)$ when
\beq \label{lb2}
n \eps^2 (1-\rho)^{-1/2} (\log n)^{1/2} = o(1).
\eeq
Under the specified parameterization, this happens exactly when $\gamma < 4\beta -2$.
\end{proof}

In the dense regime, with $\rho$ parameterized as in \eqref{dense}, we say that a test achieves the detection boundary if it is asymptotically powerful when $\gamma < 1/2 - \beta$, and in the sparse regime, with $\rho$ parameterized as in \eqref{sparse}, we say that a test achieves the detection boundary if it is asymptotically powerful when $\gamma > 4(\beta - 1/2)$.

\subsection{The covariance test}
\label{sec:cov}

Recall that the covariance test rejects for large values of $T_n := \sum_{i=1}^n X_i Y_i$, calibrated under the null where $X_1, \dots, X_n, Y_1, \dots, Y_n$ are iid standard normal.

\begin{prp}\label{prp:sum}
For the testing problem \eqref{problem1}, the covariance test achieves the detection boundary in the dense regime, while it is asymptotically powerless in the sparse regime.
\end{prp}

\begin{proof}
We divide the proof into the two regimes.

\medskip\noindent
{\em Dense regime.} \quad
Under $\cH_0$, we have
\begin{align}
\E_0(T_n) &= n\E_0(X_1 Y_1) = n\E_0(X_1)\E_0(Y_1) = 0, \\
\Var_0(T_n) &= n\Var_0(X_1 Y_1) =  n \E_0(X_1^2) \E_0(Y_1^2) = n,
\end{align}
so that, by Chebyshev's inequality,
\beq
\P_0(|T_n| \ge a_n \sqrt{n}) \to 0,
\eeq
for any sequence $(a_n)$ diverging to infinity.

Under $\cH_1$, we have
\begin{align}
\E_1(T_n) &= n\E_1(X_1 Y_1) = n \eps \rho, \\
\Var_1(T_n) &= n\Var_1(X_1 Y_1) =  n (1 + 2\eps \rho^2 - \eps^2 \rho^2) \le 3n,
\end{align}
so that, by Chebyshev's inequality,
\beq
\P_1(|T_n - n \eps \rho| \ge a_n \sqrt{n}) \to 0.
\eeq

Thus the test with rejection region $\{T_n \ge a_n \sqrt{n}\}$ is asymptotically powerful when
\beq
\sqrt{n} \eps \rho \ge 2 a_n.
\eeq
If we choose $a_n = \log n$, for example, and $\rho$ is parameterized as in \eqref{dense}, this happens for $n$ large enough when $\gamma < 1/2 - \beta$.

\medskip\noindent
{\em Sparse regime.} \quad
To prove that the covariance test is asymptotically powerless when $\beta > 1/2$, we show that, under $\cH_1$, $T_n$ converges to the same limiting distribution as under $\cH_0$.

Under $\cH_0$, by the central limit theorem,
\beq
\frac{T_n}{\sqrt{n}} \weak  \cN(0,1).
\eeq

Under $\cH_1$ the distribution of the $(X_i,Y_i)$'s (which remain iid) depends on $n$, but the condition for applying Lyapunov's central limit theorem are satisfied since
\beq
\E_1 [(X_i Y_i - \eps \rho)^4]
\le 8 (\E_1[(X_i Y_i)^4] + (\eps \rho)^4),
\eeq
with $(\eps \rho)^4 \le 1$ and
\beq
\E_1[(X_i Y_i)^4]
\le \big[\E_1(X_i^8)\E_1(Y_i^8)\big]^{1/2}
= \E(Z^8) = {\rm const},
\eeq
where $Z \sim \cN(0,1)$ and the inequality is Cauchy-Schwarz's,
while
\beq
\Var_1(X_i Y_i)
= 1 + 2\eps \rho^2 - \eps^2 \rho^2
\ge 1,
\eeq
so that the test statistic still converges weakly to a normal distribution,
\beq
\frac{T_n - \E_1(T_n)}{\sqrt{\Var_1(T_n)}} \weak \cN(0,1).
\eeq
In the present regime, we have
\beq
\E_1(T_n) = n \eps \rho, \quad
\Var_1(T_n) = n (1 + 2\eps \rho^2 - \eps^2 \rho^2),
\eeq
so that $\E_1(T_n)/\sqrt{\Var_1(T_n)} \to 0$ and $\Var_1(T_n) \sim n$, and thus we conclude by Slutsky's theorem that $T_n/\sqrt{n} \weak \cN(0,1)$.
\end{proof}

\begin{rem}
There are good reasons to consider the covariance test in this specific form since the means and variances are known.  It is worth pointing out that the Pearson correlation test, which is more standard in practice since it does not require knowledge of the means or variances, has the same asymptotic power properties.
\end{rem}

\subsection{The higher criticism test and the extremes test}
Define $U_i = (X_i - Y_i)/\sqrt{2}$, and note that
\beq
U_1, \dots, U_n \iid (1-\eps)\, \cN(0,1) + \eps\, \cN(0, 1-\rho).
\eeq
Seen through the $U_i$'s, the problem becomes that of detecting a sparse contamination where the effect is in the variance.  We recently studied this problem in detail \cite{AriasCastro:2018wr}, extending previous work by \citet{Cai:2011cb}, who considered a setting where the effect is both in the mean and variance.
Borrowing from our prior work, we consider a higher criticism test, already defined in \eqref{HC}, and an extremes test, which rejects for small values of $\min_i |U_i|$.

\begin{prp}\label{prp:hc}
For the testing problem \eqref{problem1}, the higher criticism test achieves the detection boundary in the dense and sparse regimes.
\end{prp}

\begin{proof}
Set $\sigma^2 = 1-\rho$, which is the variance of the contaminated component.  In our prior work \cite[Prop 3]{AriasCastro:2018wr}, we showed that the higher criticism test as defined in \eqref{HC} is asymptotically powerful when
\benum
\item $\sigma^2 = n^{-\gamma}$ with $\gamma > 0$ fixed such that $\gamma > 4(\beta - 1/2)$;
\item $|\sigma^2 - 1| = n^{-\gamma}$ with $\gamma > 0$ fixed such that $\gamma < 1/2 - \beta$.
\eenum
This can be directly translated into the present setting, yielding the stated result.
\end{proof}

\begin{prp}\label{prp:extremes}
For the testing problem \eqref{problem1}, the extremes test is asymptotically powerless when $\rho$ is bounded away from~1, while when $\eps$ parameterized as in \eqref{eps} and $\rho$ parameterized as in \eqref{sparse}, it is asymptotically powerful when $\gamma > 2\beta$, and asymptotically powerless when $\gamma < 2\beta$.
\end{prp}

\begin{proof}
This is also a direct corollary from our prior work our prior work \cite[Prop 2]{AriasCastro:2018wr}.
\end{proof}

Thus the extremes test is grossly suboptimal in the dense regime, while it is suboptimal in the sparse regime due to the fact that $2\beta - 4(\beta-1/2) = 2 - 2\beta > 0$.

\begin{rem}
The higher criticism and extremes tests are both based on the $U_i$'s.  This was convenient as it reduced the problem of testing for independence to the problem of testing for a change in variance (both in a contamination model). However, reducing  the original data, meaning the $(X_i, Y_i)$'s, to the $U_i$'s implies a loss of information.  Indeed, a lossless reduction would be from the $(X_i, Y_i)$'s to the $(U_i,V_i)$'s, where $V_i := (X_i+Y_i)/\sqrt{2}$, with joint distribution given in \eqref{model2}. It just turns out that ignoring the $V_i$'s does not lead to any loss in first-order asymptotic power.
\end{rem}


\subsection{Numerical experiments}
\label{sec:param_numerics}
We performed some numerical experiments to investigate the finite sample performance of the tests considered here: the likelihood ratio test, the Pearson correlation test (instead of the covariance test from a practical point of view), the extremes test, the higher criticism test, and also a plug-in version of the higher criticism test where the parameters of the bivariate normal distribution (the two means and two variances) are estimated under the null.
The sample size $n$ is set large to $n = 10^6$ in order to capture the large-sample behavior of these tests. We tried four sparsity levels, setting $\beta \in \{0.2, 0.4, 0.6, 0.8\}$.
The p-values for each test are computed as follows:
\renewcommand{\theenumi}{(\alph{enumi})}
\renewcommand{\labelenumi}{\theenumi}
\benum  \setlength{\itemsep}{0in}
\item For the {\em likelihood ratio test}, the p-values are estimated based on $10^3$ permutations.
\item For the {\em higher criticism test} and the {\em plug-in higher criticism test}, the p-values are estimated based on 200 permutations.
\item For the {\em extremes test}, we used the exact null distribution, which is available in a closed form.
\item For the {\em Pearson correlation test}, the p-values are from the limiting distribution.
\eenum

For each scenario, we repeated the process 200 times and calculated the fraction of p-values smaller than 0.05, representing the empirical power at the 0.05 level.

The results of this experiment are reported in \figref{numeric1} and are broadly consistent with the theory developed earlier in this section. Though we show that the higher criticism test is first-order comparable to the likelihood ratio test in the dense regime, even with a large sample, its power is much lower.  The Pearson correlation test does better in that regime.
The plug-in higher criticism test has a similar performance as the higher criticism test in the dense regime, while it loses some power in the moderately sparse regime, and is powerless in the very sparse regime.

\begin{figure}[ht!]
\centering	
\subfigure[$\beta=0.2$]{
\includegraphics[width=0.45\textwidth]{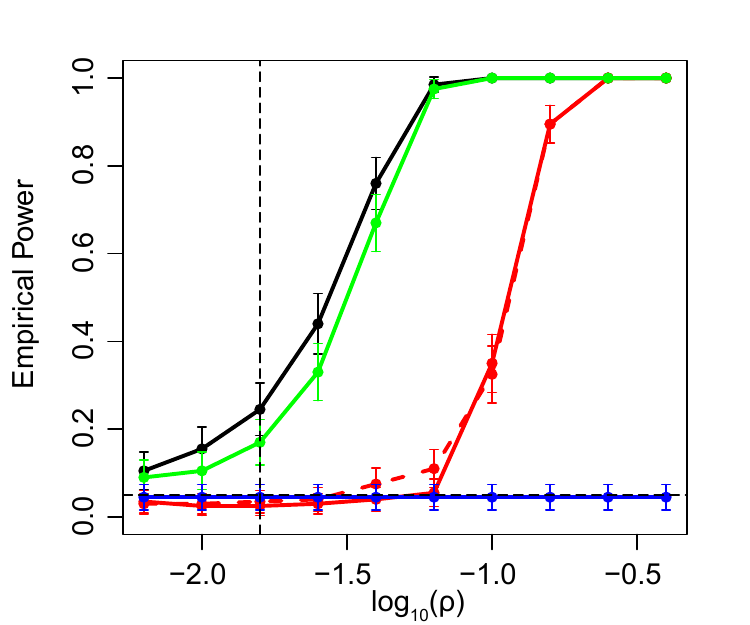}
\label{fig:subfig1}
}	
\subfigure[$\beta=0.4$]{
\includegraphics[width=0.45\textwidth]{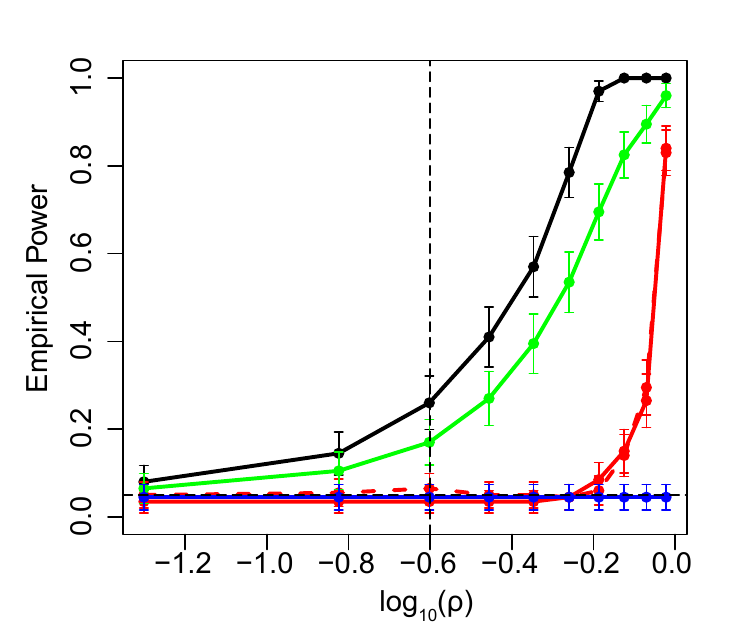}
\label{fig:subfig2}
}
\subfigure[$\beta=0.6$]{
\includegraphics[width=0.45\textwidth]{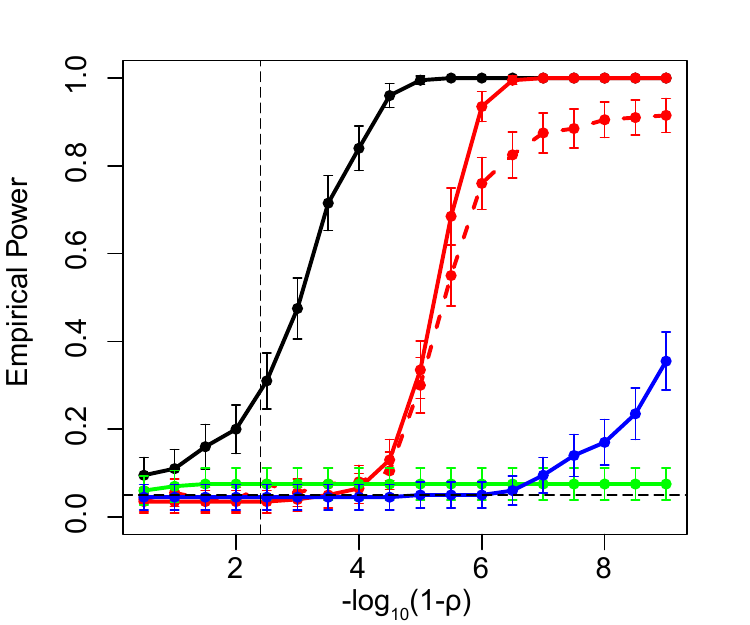}
\label{fig:subfig3}
}
\subfigure[$\beta=0.8$]{
\includegraphics[width=0.45\textwidth]{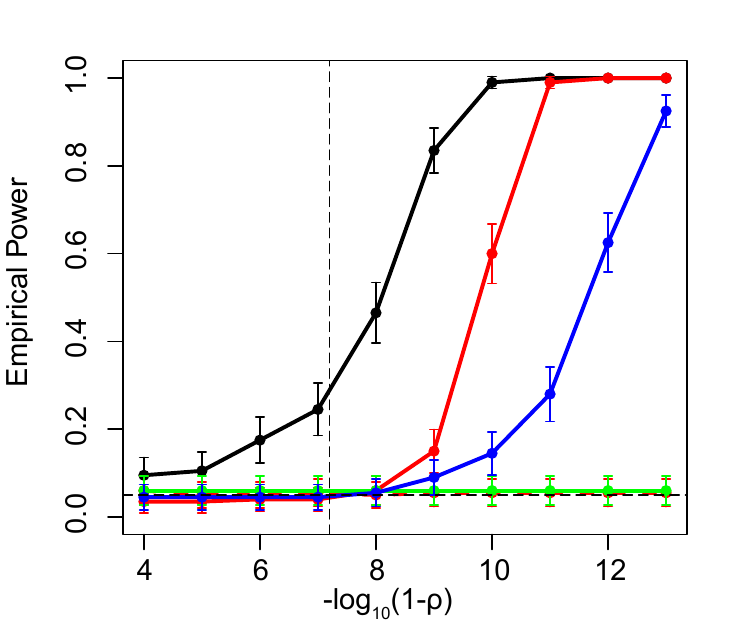}
\label{fig:subfig4}
}
\caption{Empirical power comparison with 95\% error bars for the likelihood ratio test (black), the Pearson correlation test (green), the extremes test (blue), the higher criticism test (red, solid) and the plug-in higher criticism test (red, dashed). \subref{fig:subfig1} Dense regime where $\beta=0.2$. \subref{fig:subfig2} Dense regime where $\beta=0.4$. \subref{fig:subfig3} Sparse regime where $\beta=0.6$ and $\rho \to 1$. \subref{fig:subfig4} Sparse regime where $\beta=0.8$ and $\rho \to 1$. The horizontal line marks the level (set at 0.05) and the vertical line marks the asymptotic detection boundary derived earlier.  The sample size is $n =10^6$ and the power curves and error bars are based on 200 replications.}
\label{fig:numeric1}
\end{figure}

\section{Gaussian Mixture Copula Model} \label{sec:GMCM}
In this section we turn to the Gaussian mixture copula model introduced in \eqref{model3}.
The setting is thus nonparametric, since the marginal distributions are completely unknown, and standard invariance considerations \cite[Ch 6]{lehmann2006testing} lead us to consider test procedures that are based on the ranks.  For this, we let $R_i$ denote the rank of $X_i$ among $\{X_1, \dots, X_n\}$, and similarly, we let $S_i$ denote the rank of $Y_i$ among $\{Y_1, \dots, Y_n\}$.  (The ranks are in increasing order, say.)

Although not strictly necessary, we will assume that $F$ and $G$ in \eqref{model3} are strictly increasing and continuous.  In that case, the ranks are invariant with respect to transformations of the form $(x,y) \mapsto (p(x), q(y))$ with $p$ and $q$ strictly increasing on the real line.  In particular, for the rank tests that follow, this allows us to reduce their analysis under \eqref{model3} to their analysis under \eqref{model1}.

\subsection{The covariance rank test}
The covariance rank test is the analog of the covariance test of \secref{cov}. It rejects for large values of $T_n := \sum_i R_i S_i$ (redefined).  As is well-known, this is equivalent to rejecting for large values of the Spearman rank correlation.

\begin{prp} \label{prp:spearman}
For the testing problem \eqref{problem1} under the model \eqref{model3}, the covariance rank test achieves the detection boundary in the dense regime, while it is asymptotically powerless in the sparse regime.
\end{prp}


\begin{proof}
We again divide the proof into the two regimes.

\medskip\noindent
{\em Dense regime.} \quad 	
We start by considering the null hypothesis $\cH_0$.
From \cite[Eq 3.11-3.12, Ch 11]{gibbons2003nonparametric}, we have
\begin{align}
\E_0(T_n) &= n (n+1)^2/4 = n^3/4 + O(n^2), \\
\Var_0(T_n) &= n^2 (n-1)(n+1)^2/144 \asymp n^5, \label{T-var}
\end{align}
so that, using Chebyshev's inequality,
\beq
\P_0(T_n \ge n^3/4 + a_n n^{5/2}) \to 0,
\eeq
for any sequence $(a_n)$ diverging to infinity.

We now turn to the alternative hypothesis $\cH_1$.
For convenience, we assume that the ranks run from $0$ to $n-1$.  This does not change the test procedure since $T_n = -\frac12 \sum_i (R_i - S_i)^2 + {\rm const}$, but makes the derivations somewhat less cumbersome.
In particular, we have
\begin{align}
R_i &= \sum_{j = 1}^n A_{ij}, \quad A_{ij} := \IND{X_i > X_j}, \\
S_i &= \sum_{j = 1}^n B_{ij}, \quad B_{ij} := \IND{Y_i > Y_j},
\end{align}
so that
\beq
T_n = \sum_{i=1}^n \sum_{j=1}^n \sum_{k=1}^n A_{ij} B_{ik}.
\eeq 	
For the expectation, we have
\begin{align}
\E_1(T_n)
&= n(n-1)(n-2) \E_1[A_{12} B_{13}] + O(n^2) \\
&= n^3 \E_1[A_{12} B_{13}] + O(n^2).
\end{align}
The expectation is with respect to $(X_1,Y_1), X_2, Y_3$ independent, with $(X_1,Y_1)$ drawn from the mixture \eqref{model1}, and $X_2$ and $Y_3$ standard normal.
Let $U = (X_1 - X_2)/\sqrt{2}$ and $V = (Y_1 - Y_3)/\sqrt{2}$, so that $\E_1[A_{12} B_{13}] = \P_1(U > 0, V > 0)$.
We note that $(U,V)$ is bivariate normal with standard marginals.  Moreover, when $(X_1, Y_1)$ comes from the main component, $U$ and $V$ are uncorrelated, and therefore independent; while when $(X_1, Y_1)$ comes from the contaminated component, $U$ and $V$ have correlation $\rho/2$.
Therefore,
\begin{align}
\E_1[A_{12} B_{13}]
&= (1-\eps) \Lambda(0) + \eps \Lambda(\rho/2),
\end{align}
where $\Lambda(\rho) = \P(U > 0, V > 0)$ under $(U,V) \sim \cN(0, \Sigma_\rho)$.  We immediately have $\Lambda(0) = 1/4$, and in general,\footnote{ This identity is well-known, and not hard to prove (\url{https://math.stackexchange.com/questions/255368/getting-px0-y0-for-a-bivariate-distribution}).  It also appears, for example, in \cite[Lem 1]{Xu:2013ip}.}
\beq
\Lambda(\rho) = \frac14 + \frac1{2\pi} \sin^{-1}(\rho).
\eeq
We conclude that
\begin{align}
\E_1(T_n)
&= n^3 \big[\tfrac14 + \tfrac{1}{2\pi} \eps \sin^{-1}(\rho/2)\big] + O(n^2) \\
&\ge \tfrac14 n^3 + \tfrac1{4\pi} n^3 \eps\rho  + O(n^2),
\end{align}
using the fact that $\sin^{-1}(a) \ge a$ for all $a \ge 0$.
For the variance, we start with the second moment
\begin{align}
\E_1(T_n^2)
&= n (n-1) \cdots (n-5) \E_1[A_{12} B_{13} A_{45} B_{46}] + O(n^5) \\
&= n^6 \E_1[A_{12} B_{13} A_{45} B_{46}] + O(n^5),
\end{align}
which then implies that
\begin{align}
\Var_1(T_n)
&= n^6 \E_1[A_{12} B_{13} A_{45} B_{46}] + O(n^5) - \Big[n^3 \E_1[A_{12} B_{13}] + O(n^2)\Big]^2 \\
&=O(n^5),
\end{align}
the same bound we had for $\Var_0(T_n)$.
Thus, by Chebyshev's inequality, we have
\beq
\P_1\big(T_n \le \tfrac14 n^3 + \tfrac1{4\pi} n^3 \eps \rho - a_n n^{5/2}\big) \to 0,
\eeq
for any sequence $(a_n)$ diverging to infinity.

We consider the test with rejection region $\{T_n \ge n^3/4 + a_n n^{5/2}\}$.
Our analysis implies that this test is asymptotically powerful when
\beq
n^3 \eps\rho/4\pi \ge 2 a_n n^{5/2},
\eeq
If we choose $a_n = \log n$, for example, and $\rho$ is parameterized as in \eqref{dense}, this happens for $n$ large enough when $\gamma < 1/2 - \beta$.

\medskip\noindent
{\em Sparse regime.}\quad 
To prove that the covariance rank test is asymptotically powerless when $\beta > 1/2$, similarly as the covariance test, we show that, under $\cH_1$, $T_n$ converges to the same limiting distribution as under $\cH_0$. Under $\cH_0$, we have \cite[Ch 11]{gibbons2003nonparametric},
\beq \label{h0_rank_corr}
\frac{T_n - \zeta_n}{\tau_n} \weak  \cN(0,1), \quad n \to \infty,
\eeq
where $\zeta_n := \E_0(T_n)$ and $\tau_n^2 := \Var_0(T_n)$.
We place ourselves under $\cH_1$, and show that \eqref{h0_rank_corr} continues to hold. For this we use a simple coupling. We couple $T_n$ with a new statistic $T_n'$, defined just like $T_n$, except that, for each pair $(X_i, Y_i)$ drawn from the contaminated component, we replace $Y_i$ by $Y_i' \sim \cN(0,1)$ independent of $X_i$ and any other variable.
Let $M$ denote the number of pairs drawn from the contaminated component, and note that $M$ is random, having the binomial distribution with parameters $(n, \eps)$.
It's not hard to show that $|T_n - T'_n| \le M n^2$, so that $|T_n - T'_n| = O_P(n^3 \eps)$.
And by construction, $T'_n$ has the same distribution as $T_n$ under $\cH_0$.
We use this in what follows
\beq 
\frac{T_n - \zeta_n}{\tau_n} = \frac{T_n' - \zeta_n}{\tau_n} + \frac{T_n - T_n'}{\tau_n},
\eeq
where, on the RHS, the first term converges weakly to the standard normal distribution, while the second term is $= O_P(n^3 \eps/\tau_n) = o_P(1)$, since $\eps = n^{1-\beta}$ with $\beta > 1/2$ and $\tau_n \asymp  n^{5/2}$ by \eqref{T-var}. We thus conclude that \eqref{h0_rank_corr} with an application of Slutsky's theorem. 
\end{proof}


\subsection{The higher criticism rank test}
The analog of the higher criticism test of \eqref{HC} is a higher criticism based on the pairwise differences in ranks, $D_i := |R_i - S_i|$.  To be specific, we define
\beq
\hcrank =  \underset{0 \le t \le n/2} \max \frac{\sum_{i=1}^n \IND{D_i \le t} - n u(t)}{\sqrt{ n u(t)(1-u(t))}},
\eeq
where $u(t)$ is the probability $\P_0(D_i \le t)$, which can be expressed in closed form as
\beq\label{u}
u(t)  =\frac{n^2 - (n-t)(n-t-1)}{n^2} = \frac{n(2t+1)-t(t+1)}{n^2}.
\eeq

Note that in this definition the denominator is only an approximation to the standard deviation of the numerator.  The standard deviation has a closed-form expression which can be derived from a more general result of \citet[Th 2]{hoeffding1951combinatorial}, but it is cumbersome and relatively costly to compute (although its computation is only done once for each $n$).
Also, there is a fair amount of flexibility in the choice of range of thresholds $t$ considered.  This particular choice seems to work well enough.
As any other rank test, it is calibrated by permutation (or Monte Carlo if there are no ties in the data).

\begin{thm}\label{thm:HCrank}
For the testing problem \eqref{problem1} under the model \eqref{model3}, the higher criticism rank test achieves the detection boundary in the dense and in the moderately sparse regimes.
\end{thm}


\begin{proof}
As usual, we first control the test statistic under the null, and then analyze its behavior under the alternative.

\medskip\noindent
{\em Under the null hypothesis} \\
We start with the situation under the null hypothesis $\cH_0$, where we show that $\hcrank$ is of order at most $O(\log n)$ based on the concentration inequality for randomly permuted sums.
Fixing critical value $t$, define
\beq\label{a}
a_{i,j} = \IND{|i-j| \le t}, \quad \text{for } 1 \le i,j \le n.
\eeq
Since $X$ is independent of $Y$, as we are under the null, we have that
\beq\label{Z}
\Delta(t) := \sum_{i=1}^{n} \IND{D_i \le t}
\eeq
has the same distribution as $A_n := \sum_{i=1}^{n} a_{i, \pi_n (i)}$ when $\pi_n$ is a uniformly distributed random permutation of $[n] := \{1,\cdots, n\}$.
Note that
\beq\label{eq:upper_E_an}
\E(A_n) = \frac{1}{n} \sum_{i=1}^n \sum_{j=1}^n a_{i,j} = \frac{n(2t+1)-t(t+1)}{n} = n u(t).
\eeq
By \cite[Prop 1.1]{chatterjee2007stein}, 
\begin{align}\label{eq:concentration_an}
\P(|A_n - \E(A_n)| \ge b)
&\le 2 \exp\left(- \frac{b^2}{4 \E(A_n) + 2 b}\right).
\end{align}
This implies that, for $q \ge 1$,
\begin{align}
\P_0\big(\Delta(t) \ge n u(t) + q \sqrt{n u(t) (1-u(t))}\big)
&\le 2 \exp\left(- \frac{q^2 n u(t) (1-u(t))}{4 n u(t) + 2 q \sqrt{n u(t) (1-u(t))}}\right) \\
&\le 2 \exp\left(- q/c_1\right),
\end{align}
for some other constant $c_1 > 0$, using the fact that $1/n \le u(t) \le 3/4 + 1/2n$ when $0 \le t \le n/2$, which is the range of $t$'s we are considering.
Hence, choosing $q = 2 c_1 \log n$ and using the union bound, we have
\begin{align}
\P_0(\hcrank \ge q)
&\le \sum_{t \le n/2} \P_0\big(\Delta(t) \ge n u(t) + q \sqrt{n u(t) (1-u(t))}\big) \\
&\le 2 (n/2 + 1) \exp\left(- q/c_1\right)
\asymp 1/n \to 0.
\end{align}

\medskip\noindent
{\em Under the alternative hypothesis} \\
We now consider the alternative $\cH_1$, and show that $\hcrank \gg \log n$ in probability under the stated condition.
For this, it suffices to find some $t = t_n \le n/2$ such that, for some $q = q_n \gg \log n$,
\beq\label{Zn-lb}
\Delta(t) \ge n u(t) + q \sqrt{n u(t) (1-u(t))},
\eeq
with probability tending to 1 (under $\cH_1$).

Since rank-based methods are invariant with respect to increasing transformations, in the following analysis we simply assume that $F = G = \Phi$.

\medskip\noindent
{\em Dense regime.} \quad 
Define $\hat F(x) = \frac{1}{n}\sum_{i=1}^n \IND{X_i \le x}$ and $\hat G(y) = \frac{1}{n}\sum_{i=1}^n \IND{Y_i \le y}$.  These empirical distribution functions are useful because, by definition, $R_i = n \hat F(X_i)$ and $S_i = n \hat G(Y_i)$, so that 
\begin{align}
D_i/n
&= |R_i - S_i|/n \\
&= |\hat F(X_i) - \hat G(Y_i)| \\
&\le |\hat F(X_i) - \Phi(X_i)| + |\Phi(X_i) - \Phi(Y_i)| + |\Phi(Y_i) - \hat G(Y_i)| \\
&\le \underbrace{|\Phi(X_i) - \Phi(Y_i)|}_{M_i} + \underbrace{\|\hat F - \Phi\|_\infty + \|\hat G - \Phi\|_\infty}_{K}.
\label{D-ub}
\end{align}
This gives
\beq\label{Delta_lb_dense}
\Delta(t) \ge \IND{K \le k/n} \Lambda(t), \quad  \Lambda(t) := \sum_{i=1}^n \IND{M_i \le (t-k)/n}.
\eeq

By the Dvoretzky-Kiefer-Wolfowitz (DKW) concentration inequality, there is a universal constant $c_0$ such that, for any $b \ge 0$,
\beq
\P(K \ge b) \le c_0 \exp(- n b^2/c_0).
\eeq
We choose $k = (\log n) \sqrt{n}$, and with that choice we have that $\IND{K \le k/n} = 1 - Q_n$, where $Q_n$ is Bernoulli with parameter bounded by $\eta := c_0 \exp(-(\log n)^2/c_0)$ (so that $Q_n = O_P(\eta)$).

As for the sum, the $M_i$ are iid, and for an observation $(X_i, Y_i)$ that comes from the null component, $X_i, Y_i$ are iid standard normal, while when it comes from the contaminated component, $X_i, Y_i$ are still marginally standard normal but no longer independent: $Y_i = \sqrt{1-\rho^2}\, \tilde Y_i + \rho X_i$, where $\tilde Y_i$ is independent of $X_i$ and also standard normal.
We thus have 
\beq
\P_1(M_i \le s) 
= (1-\eps) v_s(0) + \eps v_s(\rho),
\eeq
where
\beq
v_s(\rho)
:= \E[f_s(Z, \sqrt{1-\rho^2} Z' + \rho Z)],
\eeq 
where in the expectation $Z, Z'$ are iid standard normal, and $f_s(z, z') := \IND{|\Phi(z) - \Phi(z')| \le s}$ is bounded and measurable.
Elementary calculations show that $v_s(0) = 1 - (1-s)^2$, and an application of \lemref{2normals} shows that $v_s$ is infinitely differentiable, with derivative at $0$ equal to $\E[f_s(Z, Z') Z Z']$, and second derivative uniformly bounded over $[-1/2, 1/2]$ by some numerical constant, say $c_2$, independently of $s$.
Recalling that $\rho$ is small in the present regime, a Taylor development based on the above gives 
\begin{align}
v_s(\rho) 
&\ge 1 - (1-s)^2 + v_s'(0) \rho - c_2 \rho^2/2, \quad \rho \in [-1/2, 1/2].
\end{align}

In the dense regime, remember that $0 < \beta < 1/2$ and $\rho = n^{-\gamma}$.
We place ourselves above the detection boundary, meaning that we fix $\gamma < 1/2 -\beta$.
Here we choose $t = n/2$ (assumed to be an integer for convenience), let $s = (t-k)/n = 1/2-k/n$. 
We note that $v_s'(0)$ is continuous in $s$ (by dominated convergence), and because $s \to 1/2$ in our setting, we have 
\beq
v_s'(0) \to v_{1/2}'(0) = \E[\IND{|\Phi(Z) - \Phi(Z')| \le 1/2} Z Z'] =: c_1 > 0.
\eeq 
Indeed, using the fact that
\beq\notag
|\Phi(z) - \Phi(z')| \le 1/2 \iff (\Phi(z)-1/2) \vee 0 \le \Phi(z') \le (\Phi(z)+1/2) \wedge 1,
\eeq
with $\Phi(z) \le 1/2$ if and only if $z \le 0$, we have
\begin{align*} 
c_1 
&= \int_0^\infty \underbrace{\int_{\Phi^{-1}(\Phi(z)-1/2)}^\infty \phi(z') {\rm d} z'}_{> 0} \ \underbrace{\phi(z) z}_{>0} {\rm d} z + \int_{-\infty}^0 \underbrace{\int_{-\infty}^{\Phi^{-1}(\Phi(z)+1/2)} \phi(z') {\rm d} z'}_{< 0} \ \underbrace{\phi(z) z}_{< 0} {\rm d} z,
\end{align*}
where the inner integrals are positive by the fact that $\phi$ is symmetric, and the inequalities are indeed strict except when $z = 0$.

Thus, eventually (as $n \to \infty$), 
\beq
v_s(\rho) \ge 1 - (1-s)^2 + (c_1/2) \rho.
\eeq
Thus, an application of Chebyshev's inequality gives
\beq
\Lambda(n/2) 
\ge n \big[(1-\eps) v_s(0) + \eps v_s(\rho)\big] + O_P(\sqrt{n}).
\eeq

Putting everything together, we have
\begin{align}
\Delta(t) - n u(t)
&= (1 + O_P(\eta)) n \big[(1-\eps) v_s(0) + \eps v_s(\rho)\big] + O_P(\sqrt{n}) - n u(t) \\
&\ge n \big[1 - (1 - (t-k)/n)^2 - u(t)\big] + n \eps (c_1/2) \rho + O_P(n \eta) + O_P(\sqrt{n})  \\
&= n \eps (c_1/2) \rho + O_P((\log n) \sqrt{n}), 
\end{align}
using the fact that $\eta = o(1/n^2)$.
For \eqref{Zn-lb} to hold it thus suffices that $n \eps \rho \gg (\log n) \sqrt{n}$, which is the case since $n \eps \rho = n^{1-\beta-\gamma}$ with $1-\beta-\gamma > 1/2$.

\medskip\noindent
{\em Moderately sparse regime.} \quad
Let $I_0$ and $I_1$ index the observations coming from the null and contaminated components, respectively.
We have
\beq\label{Delta_decomposition}
\Delta(t)
= \sum_{i \in I_0} \IND{D_i \le t} + \sum_{i \in I_1} \IND{D_i \le t}
=: \Delta_0(t) + \Delta_1(t).
\eeq
We lower bound both terms on the right-hand side, starting with $\Delta_0(t)$. To do this, we consider a slightly smaller threshold, specifically 
$t_0 = (1 - \omega) t$ with $\omega = o(1)$ specified below, and compare $\Delta_0(t)$ with $\Delta^0(t_0) := \sum_{i \in I_0} \IND{D^0_i \leq t_0}$, where $D^0_i := |R^0_i - S^0_i|$ with $R^0_i$ denoting the rank of $X_i$ among $\{X_j: j\in I_0\}$ and $S^0_i$ denoting the rank of $Y_i$ among $\{Y_j: j\in I_0\}$.  
Conditional on $|I_0| = n_0$, $\Delta^0(t_0)$ has the same distribution as $\Delta(t_0)$ in \eqref{Z} under the null hypothesis but with $n$ replaced by $n_0$, so that from~\eqref{eq:upper_E_an} we deduce that it has expectation 
\beq
\mu := (n_0(2t_0+1)-t_0(t_0+1))/n_0,
\eeq 
and from~\eqref{eq:concentration_an} that
\beq
\Delta^0(t_0)\geq \mu - 8(\log n) \sqrt{\mu \vee \log n}
\eeq
with probability at least $1-2/n$ when $n$ is large enough. (Again, this is conditional on $|I_0| = n_0$.)
Because $\eps \ll n^{-1/2}$ in the present regime, we have $|I_0| \ge n - (\log n) \sqrt{n}$ with probability at least $1-1/n$ when $n$ is large enough. Also, we will choose $t$ below such that $\sqrt{n} \ll t \ll n$, and $\omega$ such that $\omega \ll 1$, so that $t_0 \sim t$.
Together, this implies that 
\beq
\Delta^0(t_0) 
\ge 2 t_0 + 1 - \frac{t_0 (t_0+1)}{n -(\log n) \sqrt{n}} - 8(\log n) \sqrt{2 t_0 + 1}
= 2 t_0 - \frac{t_0^2}n - O((\log n)\sqrt{t}),
\eeq
eventually, with probability at least $1-3/n$.

We now claim that, with probability tending to 1, $\Delta_{0}(t) \ge \Delta^0(t_0)$. 
Indeed, by definition of the ranks $R_i$ and modified ranks $R^0_i$, we have
\beq
R_i- R^0_i 
= \sum_{j\in I_1} \IND{X_j \leq  X_i} 
= |I_1| \hat F_1(X_i),
\eeq
where $\hat F_1(x) := \frac{1}{|I_1|}\sum_{j \in I_1} \IND{X_j \le x}$ is the empirical distribution function associated with the contaminated $X$ observations.
In particular, when $|I_0| = n_0$, so that $|I_1| = n-n_0 =: n_1$, we have
\beq
\big|R_i- R^0_i - n_1 \Phi(X_i)\big| \le n_1 \|\hat F_1 - \Phi\|_\infty,
\eeq
valid for all $i \in I_0$.
At the same time, and with analogous notation, we also have
\beq
\big|S_i- S^0_i - n_1 \Phi(Y_i)\big| \le n_1 \|\hat G_1 - \Phi\|_\infty,
\eeq
valid for all $i \in I_0$.
Combining these, we obtain
\beq
\underbrace{|R_i - S_i|}_{D_i}
\le \underbrace{|R^0_i - S^0_i|}_{D^0_i} + n_1 |\Phi(X_i) - \Phi(Y_i)| + n_1  \underbrace{\big(\|\hat F_1 - \Phi\|_\infty + \|\hat G_1 - \Phi\|_\infty\big)}_{=: K_1},
\eeq
valid for all $i \in I_0$.
Letting $\hat F_0$ denote the empirical distribution function of $\{X_i : i \in I_0\}$ and $\hat G_0$ denote that of $\{Y_i : i \in I_0\}$, we have
\beq\label{Phi-Phi_0}
|\Phi(X_i) - \Phi(Y_i)| 
\le \underbrace{|\hat F_0(X_i) - \hat G_0(Y_i)|}_{D^0_i/n_0} + \underbrace{\|\hat F_0 - \Phi\|_\infty + \|\hat G_0 - \Phi\|_\infty}_{=: K_0},
\eeq
valid for all $i \in I_0$. Note that this is conditional on $|I_0| = n_0$ and that the distributions of $K_0$ and $K_1$ depend (implicitly) on $n_0$ (and $n_1$). 
We conclude that, conditional on $|I_0| = n_0$, for any $i \in I_0$,
\beq\label{D_D0}
D_i \le (n/n_0) D^0_i + n_1 (K_0 + K_1).
\eeq
Applying the DKW inequality with the tight constant, we have that $K_0 \le (\log n)/\sqrt{n_0}$ and $K_1 \le (\log n)/\sqrt{n_1}$ with probability at least $1-2/n$ when $n$ is large enough, and when this is the case, $D_i \le (n/n_0) D^0_i + 2 (\log n) \sqrt{n_1}$, assuming that $n_0 \ge n_1$.
This is given $|I_0| = n_0$ and (therefore) $|I_1| = n_1$, and we also know that $|I_0| \ge n - (\log n) \sqrt{n}$ and $|I_1| \le 2 n \eps$ with probability at least $1-1/n$ when $n$ is large enough.  (We are using that $|I_1| \sim \Bin(n, \eps)$ with $n \eps = n^{1-\beta}$ with $\beta < 1$.)
Hence, with probability at least $1-3/n$,
\beq
D_i \le \frac{n D^0_i}{n - (\log n) \sqrt{n}} + 2 (\log n) \sqrt{2 n \eps},
\eeq
for any $i \in I_0$.
In particular, if we choose $\omega = (\log n)^2 \max\big(1/\sqrt{n}, \sqrt{n \eps}/t\big)$, then, with probability at least $1 - 2/n$ when $n$ is large enough, $D^0_i \le t_0$ implies that $D_i \le t$ for any $i \in I_0$, implying that $\Delta_0(t) \ge \Delta^0(t_0)$.

We thus conclude that
\beq
\Delta_0(t) \ge 2 t_0 - t_0^2/n - O_P((\log n)\sqrt{t}).
\eeq

As for $\Delta_1(t)$, as in \eqref{Delta_lb_dense}, we have
\beq
\Delta_1(t) \ge \IND{K \le k/n} \Lambda_1(t), \quad  \Lambda_1(t) := \sum_{i \in I_1} \IND{M_i \le (t-k)/n}.
\eeq
We choose $k = (\log n) \sqrt{n}$ as we did before, so that $\IND{K \le k/n} = 1 + O_P(\eta)$, with the same $\eta$ defined previously. 
As for the sum, $\Lambda_1(t)$ has the same distribution as $\sum_{i=1}^B \IND{\tilde M_i \le (t-k)/n}$, with $B$ binomial with parameters $(n, \eps)$ and $\tilde M_i
= |\Phi(\tilde X_i) - \Phi(\tilde Y_i)|$ with $(\tilde X_i, \tilde Y_i)$ iid normal with standard normal marginals and correlation $\rho$.  In particular,
\beq
\tilde M_i
\le \tfrac1{\sqrt{2\pi}} |\tilde X_i - \tilde Y_i| =: \tfrac1{\sqrt{\pi}} |\tilde U_i|,
\eeq
by the fact that $\Phi$ has derivative bounded by $1/\sqrt{2\pi}$ everywhere, and where $\tilde U_i \sim \cN(0, 1-\rho)$, and simple calculations give
\beq
v(s)
:= \P(\tilde M_i \le s) 
\ge \Psi\left(\frac{\sqrt{\pi} s}{ \sqrt{1-\rho}}\right) =: \lambda(s), \quad s \in [0,1]. 
\eeq
We thus have
\beq
\E_1(\Lambda_1(t)) = n \eps v((t-k)/n),
\eeq
and 
\beq
\Var_1(\Lambda_1(t)) 
= \Var(B) v((t-k)/n)^2 + \E(B) v((t-k)/n)
\le 2 n \eps v((t-k)/n),
\eeq
and applying Chebyshev's inequality, we thus have 
\begin{align}
\Lambda_1(t) 
&= n\eps v((t-k)/n) + O(\sqrt{n\eps v((t-k)/n)}) \\
&\ge (1+o_P(1)) n \eps \lambda((t-k)/n)),
\end{align}
as long as the right-hand side diverges.

In the moderately sparse regime, remember that $1/2 < \beta < 3/4$ and $\rho = 1 - n^{-\gamma}$.
We place ourselves just above the detection boundary, meaning that we fix $\gamma > 4 (\beta -1/2)$.
We focus on the harder sub-case where, in addition, $\gamma < 2\beta$.  In that case, we can fix $a$ such that $1/2 > a > \gamma/2$ and $1/2-\beta+\gamma/2 - a/2 > 0$, and set $t = \lfloor n^{1-a} \rfloor$.  Note that such a real number $a$ exists, and that $t \le n/2$ with $t \gg k$. 
We also have $n \eps = n^{1-\beta}$ and $u(t) \asymp t/n \asymp n^{-a}$, as well as 
\beq
\lambda((t-k)/n) = \Psi\left(\frac{\sqrt{\pi} (t-k)}{n \sqrt{1-\rho}}\right) \asymp n^{\gamma/2 - a}, \quad \text{since } \frac{\sqrt{\pi} (t-k)}{n \sqrt{1-\rho}} \asymp n^{\gamma/2 - a} \to 0,
\eeq
and $\Psi$ is differentiable at $0$ with positive derivative.  
In particular, $n \eps \lambda((t-k)/n) \asymp n^{1-\beta+\gamma/2-a} \to \infty$.
Putting everything together, we have
\begin{align*}
\Delta(t) - n u(t)
&\ge 2 t_0 - t_0^2/n - O_P((\log n)\sqrt{t}) + (1 + o_P(1)) n \eps \lambda((t-k)/n) - \big(2t+1 - t(t+1)/n\big) \\
&= - O_P((\log n) \sqrt{t}) + n \eps (1 + o_P(1)) \lambda((t-k)/n),
\end{align*}
after some simplifications, using the definition of $\omega$ above and the fact that $\sqrt{n} \ll t \ll n$.
For \eqref{Zn-lb} to hold, it is thus enough to have $n \eps \lambda((t-k)/n) \gg (\log n) \sqrt{t}$, which is the case since  
\beq
\frac{n \eps \lambda((t-k)/n)}{\sqrt{t}} 
\asymp \frac{n^{1-\beta+\gamma/2-a}}{n^{1/2-a/2}} 
= n^{1/2-\beta+\gamma/2-a/2},
\eeq
with $1/2-\beta+\gamma/2-a/2 > 0$ by our choice of $a$.
\end{proof}

\begin{lem}\label{lem:2normals}
Let $A, B$ be iid standard normal, and for $f\colon \bbR^2 \to [0,1]$ measurable and $r \in [-1,1]$, define $\Gamma_f(r) = \E[f(A, \sqrt{1-r^2} B + r A)]$. Then $\Gamma_f$ is infinitely differentiable, with $\Gamma_f'(0) = \E[f(A, B) AB]$, and with $\sup_{|r| \le 1/2} |\Gamma_f''(r)|$ bounded by some numerical constant (independent of $f$).
\end{lem}

\begin{proof}
We have
\beq
\Gamma_f(r)
= \int_{-\infty}^\infty \int_{-\infty}^\infty f(a, b) \phi(a,b;r) {\rm d}a {\rm d}b,
\eeq
where
\beq
\phi(a,b;r) := \frac{\exp\big[-(a^2 - 2rab + b^2)/(2-2r^2)\big]}{2\pi\sqrt{1-r^2}}.
\eeq
An application of the dominated convergence theorem allows us to differentiate under the integral at will.
In particular,
\beq
\Gamma_f^{(k)}(r) = \int_{-\infty}^\infty \int_{-\infty}^\infty f(a, b) \partial_r^k \phi(a,b;r) {\rm d}a {\rm d}b,
\eeq
Elementary calculations show that $\partial_r \phi(a,b;0) = (2\pi)^{-1} ab \exp[-(a^2+b^2)/2]$.
We also obtain
\beq
|\Gamma_f''(r)| \le \int_{-\infty}^\infty \int_{-\infty}^\infty |\partial_r^2 \phi(a,b;r)| {\rm d}a {\rm d}b,
\eeq
which is easily seen to uniformly bounded for $|r| \le 1/2$. 
\end{proof}

It is natural to wonder whether the higher criticism rank test has some power in the very sparse regime. The following indicates that it is powerless in that regime.

\begin{prp} \label{prp:Delta_very_sparse}
Consider the very sparse regime in the most extreme case where $\rho = 1$. In that setting, any test that rejects for large values of $\Delta(t) := \sum_{i=1}^n \IND{D_i \le t}$ (where the threshold $t$ is allowed to vary with $n$) is asymptotically powerless.
\end{prp}

\begin{proof}
By a compactness argument, we may assume that either $t \to \infty$ or $t$ is constant (as $n$ varies). We start with the former and address the latter at the end.
We focus on the case where $t \ll n$, as the case where $t \asymp n$ can be dealt with in a very similar fashion.  

\medskip\noindent
{\em Under the null hypothesis} \\
We first consider the behavior of $\Delta(t)$ under the null hypothesis, and argue that $\Delta(t)$ is asymptotically normally distributed. 
This is based on an application of a combinatorial central limit theorem due to \citet{hoeffding1951combinatorial}.
Remember that under $\cH_0$, $\Delta(t)$ has the distribution of $A_n = \sum_{i=1}^{n} a_{i, \pi_n (i)}$ when $\pi_n$ is a uniformly distributed random permutation of $[n]$ and $a_{i,j} = \IND{|i-j| \le t}$.
We saw that
\beq
\E(A_n) = \frac{1}{n} \sum_{i=1}^n \sum_{j=1}^n a_{i,j} = \frac{n(2t+1)-t(t+1)}{n} = n u(t),
\eeq
and, as derived in \cite{hoeffding1951combinatorial}, we also have 
\beq
\Var(A_n) =  \frac{1}{n-1} \sum_{i=1}^n \sum_{j=1}^n d_{i,j}^2,
\eeq
where 
\beq
d_{i,j} = a_{i,j} - \frac{1}{n} \sum_{g=1}^n a_{g,j} - \frac{1}{n} \sum_{h=1}^n a_{i,h} + \frac{1}{n^2} \sum_{g=1}^n \sum_{h=1}^n a_{g,h}.
\eeq
\cite[Th 3]{hoeffding1951combinatorial} implies that $A_n$ is asymptotically normal when
\beq
\frac{\max_{\, i, j \in [n]} d_{i,j}^2}{\frac1{n^2} \sum_{i \in [n]} \sum_{j \in [n]} d_{i,j}^2} \to 0.
\eeq
Elementary but somewhat tedious calculations yield that this is the case if and only if $t \to \infty$, which we assume. 
Further elementary calculations, in part similar to some appearing in the proof of \thmref{HCrank}, yield that 
\beq
\frac{\Var(A_n)}{n u(t) (1- u(t))} \to 1,
\eeq
We thus have, under the null hypothesis,
\beq
\frac{\Delta(t) - n u(t)}{\sqrt{n u(t) (1-u(t))}} \weak \cN(0,1),
\eeq
and therefore, together with the fact that $1 \ll t \ll n$, we conclude that
\beq\label{Delta_normal}
\frac{\Delta(t) - 2t + t^2/n}{\sqrt{2t}} \weak \cN(0,1),
\eeq
again under the null hypothesis.

\medskip\noindent
{\em Under the alternative hypothesis} \\
We now consider the alternative, again in the very sparse regime and in the most advantageous case where $\rho = 1$, and show that the same weak limit holds. For this, we follow the arguments of the proof of \thmref{HCrank} in the moderately sparse regime, although in the reverse direction so-to-speak. We use the same notation.

Starting from the decomposition \eqref{Delta_decomposition}, we have
\begin{align}
\frac{\Delta(t) - 2t + t^2/n}{\sqrt{2t}}
&= \frac{\Delta_0(t) - 2t + t^2/n}{\sqrt{2t}} + \frac{\Delta_1(t)}{\sqrt{2t}}.
\label{Delta_very_sparse_1}
\end{align}
In what follows, we first show that the first term on the RHS is asymptotically standard normal, and then we show that the second term converges to 0 in probability.

\medskip\noindent
{\em First term in \eqref{Delta_very_sparse_1}.} \quad
For $i \in I_0$, as in \eqref{D_D0} but in reverse, we have 
\begin{align}
D^0_i 
&\le (1 + |I_1|/n) D_i + |I_1| (K_0 + K_1) \\
&\le \big(1 + \eps + (\log n) \sqrt{\eps/n}\big) D_i + (\log n) \sqrt{n \eps},
\end{align}
with probability tending to~1 uniformly over $i \in I_0$.
Assuming this is true, then $D_i \le t$ implies that 
\begin{align}
D^0_i 
&\le \big(1 + \eps + (\log n) \sqrt{\eps/n}\big) t + (\log n) \sqrt{n \eps} \\
&\le t_0 := (1 + \eps) t + 2 (\log n) \sqrt{n \eps}. 
\end{align}
Hence, with probability tending to~1,
\beq
\Delta_0(t) \le \Delta^0(t_0).
\eeq
As before, conditional on $|I_0| = n_0$, $\Delta^0(t_0)$ has the same distribution as $\Delta(t_0)$ in \eqref{Z} under the null hypothesis but with $n$ replaced by $n_0$. This, the fact that $|I_0| \ge n - O_P(\sqrt{n})$, and \eqref{Delta_normal}, implies that
\beq
\frac{\Delta^0(t_0) - 2t_0 + t_0^2/n}{\sqrt{2t_0}} \weak \cN(0,1).
\eeq
We used the fact that $t_0^2/n \le t_0^2/|I_0| \le t_0^2/(n - O(\sqrt{n}))$, which implies that 
\beq
\frac{t_0^2/|I_0|}{\sqrt{t_0}} 
= \frac{t_0^2/n}{\sqrt{t_0}} + \underbrace{O(t_0 \sqrt{t_0}/n\sqrt{n})}_{o(1)},  
\eeq
where the $O$ term is $o(1)$ by the fact that $t_0/n = o(1)$.
Continuing, with probability tending to~1, we have
\begin{align}
\frac{\Delta_0(t) - 2t + t^2/n}{\sqrt{2t}}
&\le \frac{\Delta^0(t_0) - 2t + t^2/n}{\sqrt{2t}} \\
&= \sqrt{t_0/t}\, \frac{\Delta^0(t_0) - 2t_0 + t_0^2/n}{\sqrt{2t_0}} + \frac{2t_0 - t_0^2/n -2t + t^2/n}{\sqrt{2t}} \\
&\weak \cN(0,1), \label{normal_upper_limit}
\end{align}
whenever $t_0/t \to 1$ and $(t_0 - t)/\sqrt{t} \to 0$ (using the fact that $t \le t_0 \ll n$).
This is the case exactly when $t \gg (\log n)^2 n \eps$. 

We now consider the complementary case. In fact, what follows applies when $t \le \sqrt{n}$. 
We use a slightly different strategy.
Recall that, for $i \in I_0$,
\begin{align}
R_i- R^0_i &= \sum_{j\in I_1} \IND{X_j \leq X_i}, \\
S_i- S^0_i &= \sum_{j\in I_1} \IND{Y_j \leq Y_i},
\end{align}
and combined with the triangle inequality, and recalling that $X_j = Y_j$ when $j \in I_1$, we have
\beq
|D_i - D^0_i| 
\le W_i := \sum_{j \in I_1} \IND{X_i \wedge Y_i \le X_j \le X_i \vee Y_i}.
\eeq
Consider the event 
\beq
\Omega = \big\{|I_1| \le 2 n\eps, K_0 \le (\log n)/\sqrt{n}\big\},
\eeq 
which happens with probability tending to one.
Given $\Omega$, we have
\begin{align}
\{D_i \le t\} 
&= \{D_i \le t, D^0_i \le t\} \cup \{D_i \le t, D^0_i > t\} \\ 
&\subset \{D^0_i \le t\} \cup \{W_i \ge D^0_i - t, 2 n \eps + t \ge D^0_i > t\},
\end{align}
using the fact that $D^0_i \le D_i + |I_1|$, so that
\begin{align}
\Delta_0(t) 
&\le \Delta^0(t) + \sum_{i \in I_0} \IND{W_i \ge D^0_i - t, 2 n \eps + t \ge D^0_i > t}.
\label{Delta_Delta0_W}
\end{align}
Given $\{(X_k, Y_k) : k \in I_0\}$, and conditional on $(|I_0|,|I_1|) = (n_0, n_1)$, $W_i$ is binomial with parameters $n_1$ and $P_i := |\Phi(X_i) - \Phi(Y_i)|$. As in \eqref{Phi-Phi_0}, the latter is bounded by $D^0_i/n + K_0$, which itself is bounded (eventually) by $2(\log n)/\sqrt{n}$ under $\Omega$ when $D^0_i = d$ with $d \le t + 2 n \eps$ (since we work under the assumption that $t \le \sqrt{n}$).
Thus, for such a $d$, eventually,
\begin{align}
\P(W_i \ge w \mid \Omega, D^0_i = d)
&\le \E\Big[\P\big(W_i \ge w \mid \Omega, D^0_i = d, (X_k, Y_k)_{k \in I_0}\big)\Big] \\
&\le 2 \P\big(W_i \ge w \mid P_i \le 2 (\log n)/\sqrt{n}\big) \\
&\le 2\, {\rm Prob}\big(\Bin(2n\eps, 2(\log n)/\sqrt{n}) \ge w \big) \\
&\le c_0 (n \eps \times (\log n)/\sqrt{n})^w,
\end{align}
where $c_0$ is a universal constant.
The factor of 2 in the second inequality comes from de-conditioning from $\{K_0 \le (\log n)/\sqrt{n}$.
In the last line we used the fact that ${\rm Prob}(\Bin(m, q) \ge k) \le \binom{m}{k} q^k$, referred to as the Gin\'e--Zinn inequality in \cite{dasgupta2008asymptotic}.
We also have
\begin{align}
\P(D^0_i = d \mid \Omega)
\le 2 \P(D^0_i = d \mid |I_1| \le 2 n \eps)
\le 2 \frac{2}{n - 2 n \eps}
\le \frac{5}n,
\end{align}
eventually, using the fact that $\P(D^0_i = d \mid |I_0| = n_0) \le 2/n_0$.
Together, this yields
\begin{align}
&\P\big(W_i \ge D^0_i - t, 2 n \eps + t \ge D^0_i > t \mid \Omega \big) \\
&\le \sum_{d \ge t+1}^{t+2 \lfloor n \eps \rfloor} \P\big(W_i \ge d - t \mid \Omega, D^0_i = d) \times \P(D^0_i = d \mid \Omega) \\
&\le c_1 \sum_{d \ge t+1}^{t+2 \lfloor n \eps \rfloor} ((\log n) \sqrt{n} \eps)^{d-t} \times \frac{1}n, \\
&\le c_1 \times \frac{2}n \times ((\log n) \sqrt{n} \eps).
\end{align}
Hence, the second term on the RHS of \eqref{Delta_Delta0_W} has expectation of order at most $n$ times the last term in our last derivations, meaning of order at most $(\log n)\sqrt{n} \eps = o(1)$. Since that term is integer-valued, this implies that $\Delta_0(t) \le \Delta^0(t)$ with probability tending to one.
In particular, \eqref{normal_upper_limit} applies.

\medskip\noindent
{\em Second term in \eqref{Delta_very_sparse_1}.} \quad
Consider $i \in I_1$.
Because $\rho = 1$, we have  $X_i = Y_i$, and conditional on $X_i = z$, $R_i - 1$ and $S_i -1$ are iid with distribution $\Bin(n-1, p)$ where $p := \Phi(z)$. In particular, $D_i$ has the distribution of $|U - V|$ where $U$ and $V$ are iid with distribution $\Bin(n-1, P)$ and $P \sim \Unif(0,1)$. Let $u_2(t)$ denote the probability that $D_i \le t$. We want to bound $u_2(t)$ from above.

For $p \in [0,1]$, define $g(p)$ as the probability that $|U -V| \le t$ when $U$ and $V$ are iid $\Bin(n-1, p)$, and note that $u_2(t) = \int_0^1 g(p) {\rm d}p$.
Define $\sigma^2 = 2 (n-1) p(1-p)$, which is the variance of $U-V$, and also $h(a) = \P((U-V)/\sigma \le a)$.
Using the fact that $U - V$ is integer valued, we have
\begin{align}
g(p) 
= h(t/\sigma) - h(-(t+1)/\sigma)
\le \Phi(t/\sigma) - \Phi(-(t+1)/\sigma) + 2 \|h - \Phi\|_\infty.
\end{align}
Where $\Phi$ is the standard normal distribution function. 
Because $\Phi$ has derivative bounded by $1/\sqrt{2\pi}$ everywhere, the first term on the RHS is $= O(t/\sigma)$.
For the second term, we use the Berry--Esseen inequality (seeing $U$ and $V$, each, as the sum of $n-1$ iid ${\rm Ber}(p)$ random variables), to get that it is $= O(1/\sigma)$. Therefore, since $t \ge 1$, there is a universal constant $c_0$ such that $g(p) \le c_0 t/\sigma$. Of course, being a probability, we also have $g(p) \le 1$.
Hence,
\beq
u_2(t) 
= \int_0^1 g(p) {\rm d}p
\le \int_0^1 \left(1 \wedge  \frac{c_0 t}{2(n-1)p(1-p)}\right) {\rm d}p
\asymp 1 \wedge t/\sqrt{n}.
\eeq  

Now, by Markov's inequality, and the fact that $|I_1|$ is binomial with parameters $(n, \eps)$, the second term in \eqref{Delta_very_sparse_1} is 
\beq
= \frac{O_P(n \eps) O_P(u_2(t))}{\sqrt{n u(t) (1-u(t))}}
\asymp \frac{n^{1-\beta} (1 \wedge t/\sqrt{n})}{\sqrt{t}}
\asymp n^{1-\beta} t^{-1/2} \wedge n^{1/2-\beta} t^{1/2}
\to 0, 
\eeq
for any choice of $t$ when $\beta > 3/4$ (very sparse regime).

\medskip\noindent
{\em Special case: $t$ constant.}
When $t$ is constant, the null distribution of $\Delta(t)$ is known to converge to the Poisson distribution of mean $2t+1$. (See \cite[Exa 1.3]{arratia1990poisson}, which is only cosmetically different.) 
The control under the alternative can be secured in exactly the same way. In particular, it holds that $\Delta_0(t) \le \Delta^0(t)$ with probability tending to one, with $\Delta^0(t)$ having the same asymptotic distribution (Poisson with mean $2t+1$).
\end{proof}

\subsection{Numerical experiments}
We consider the same setting as in \secref{GMM} and compare the two nonparametric tests, the covariance rank test and the higher criticism rank test, to the parametric tests.
The p-values for the higher criticism rank test are obtained based on $10^5$ permutations, while the p-values for the covariance rank test are taken from the limiting distribution based on its correspondence with the Spearman rank correlation.

The results are presented in \figref{numeric2}.
In finite samples, the higher criticism rank test exhibits substantially more power than the higher criticism in the dense and moderately sparse regime. We have no good explanation for this rather surprising phenomenon. However, the higher criticism rank test has no power in the very sparse regime, and neither does the covariance rank test.

\begin{figure}[ht!]
\centering	
\subfigure[$\beta=0.2$]{
\includegraphics[width=0.45\textwidth]{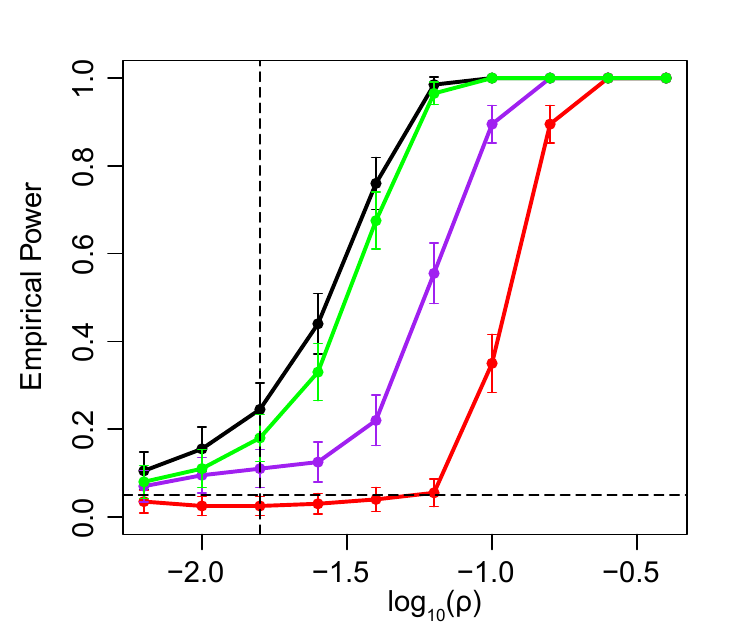}
\label{fig:subfig5}
}	
\subfigure[$\beta=0.4$]{
\includegraphics[width=0.45\textwidth]{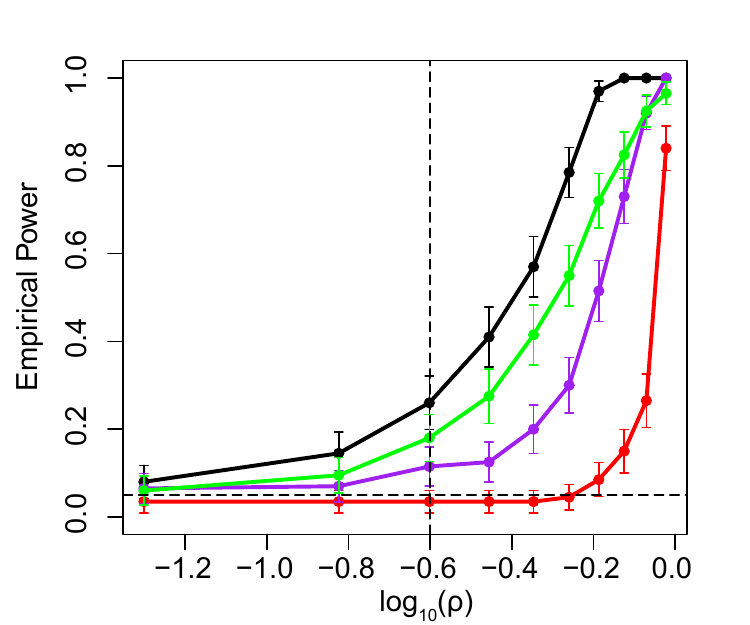}
\label{fig:subfig6}
}
\subfigure[$\beta=0.6$]{
\includegraphics[width=0.45\textwidth]{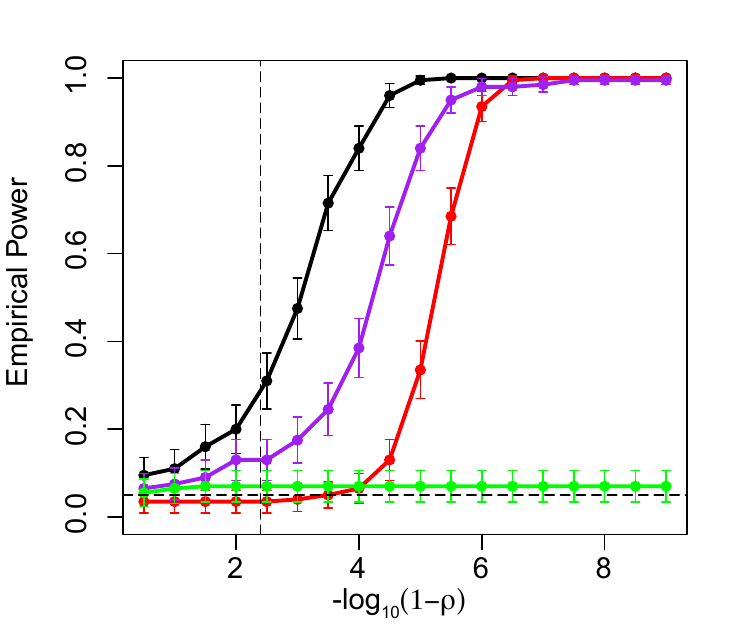}
\label{fig:subfig7}
}
\subfigure[$\beta=0.8$]{
\includegraphics[width=0.45\textwidth]{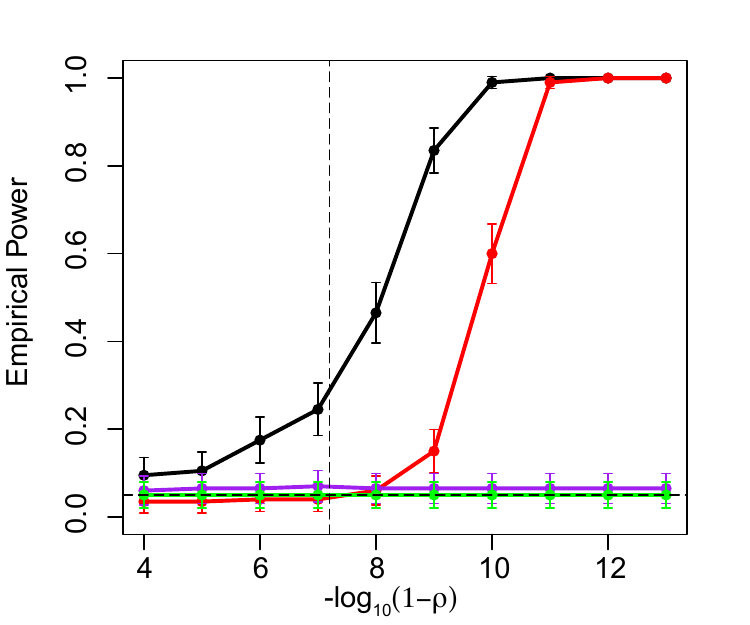}
\label{fig:subfig8}
}	
\caption{Empirical power comparison with 95\% error bars for the likelihood ratio test (black), the covariance rank test (green), the higher criticism test (red) and the higher criticism rank test (purple). \subref{fig:subfig5} Dense regime where $\beta=0.2$. \subref{fig:subfig6} Dense regime where $\beta=0.4$. \subref{fig:subfig7} Sparse regime where $\beta=0.6$ and $\rho \to 1$. \subref{fig:subfig8} Sparse regime where $\beta=0.8$ and $\rho \to 1$. The horizontal line marks the level (set at 0.05) and the vertical line marks the asymptotic detection boundary derived earlier. The sample size is $n =10^6$ and the power curves and error bars are based on 200 replications.}
\label{fig:numeric2}
\end{figure}

\section{Discussion}
\label{sec:discussion}

\paragraph{The power residing in the $V_i$}
In \prpref{hc} we established that the higher criticism test based on $U_1, \dots, U_n$ achieves the detection boundary in the Gaussian mixture model.  It is natural, however, to ask whether one could do better in finite samples by also utilizing $V_1, \dots, V_n$.  We performed some side experiments to quantify this by comparing the full LRT, meaning the LRT based on $(U_1, V_1), \dots, (U_n, V_n)$, the LRT based on $U_1, \dots, U_n$ only, and the LRT based on $V_1, \dots, V_n$ only.  We did so in the same parametric setting of \secref{param_numerics}.
The results are reported in \figref{supp}, and can be to some extent anticipated from our previous work \cite{AriasCastro:2018wr}.  In a nutshell, in the dense regime, what matters is the deviation of the variance from 1, and this is felt by all tests, so that the $U$-LRT and the $V$-LRT are seen to be also as powerful as the full LRT.  In the sparse regime, however, we can see that the $V$-LRT has essentially no power.  This is due to the fact that the $V_i$'s in that case have variance $1+\rho$, which is bounded from above by 2, so that no test depending on the $V_i$'s can have any power as we show in \cite{AriasCastro:2018wr}.  The $U$-LRT, which we know to be asymptotically optimal to first order, remains competitive, although now clearly less powerful than the full LRT.

\begin{figure}[ht!]
\centering	
\subfigure[$\beta=0.4$]{
\includegraphics[width=0.45\textwidth]{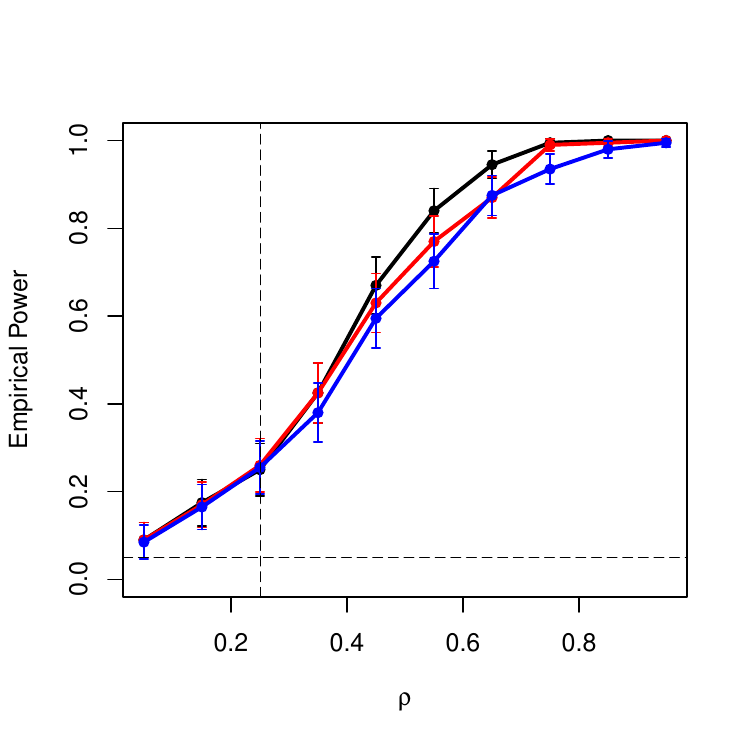}
\label{fig:subfig9}
}	
\subfigure[$\beta=0.6$]{
\includegraphics[width=0.45\textwidth]{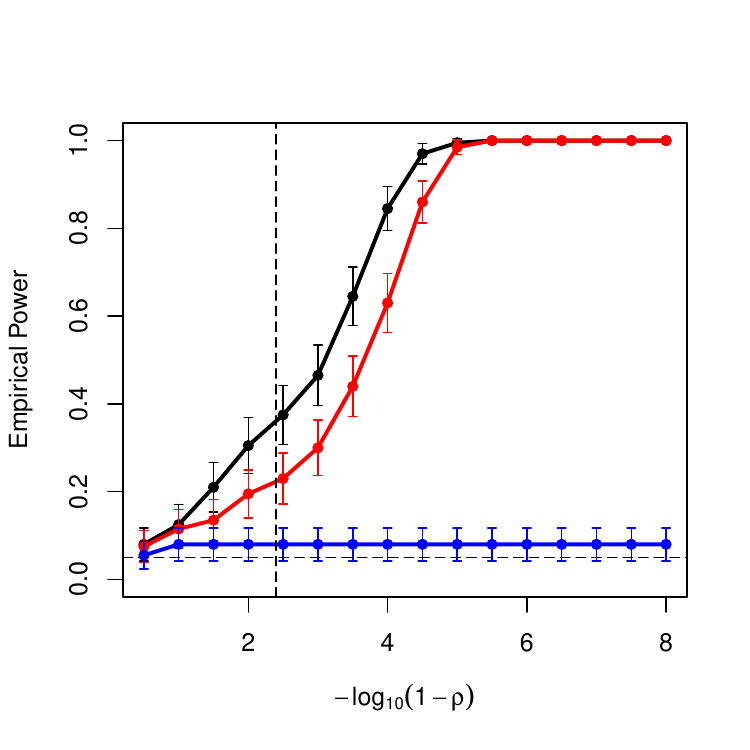}
\label{fig:subfig10}
}
\caption{Empirical power comparison with 95\% error bars for the full LRT (black), the $U$-LRT (red) and the $V$-LRT (blue). \subref{fig:subfig9} Dense regime where $\beta=0.4$. \subref{fig:subfig10} Sparse regime where $\beta=0.6$ and $\rho \to 1$. The horizontal line marks the level (set at 0.05) and the vertical line marks the asymptotic detection boundary derived earlier. The sample size is $n =10^6$ and the power curves and error bars are based on 200 replications.} 
\label{fig:supp}
\end{figure}

\paragraph{The power of rank tests in the very sparse regime}
In \prpref{Delta_very_sparse} we argued, we hope convincingly, that no test that resembles the higher criticism rank test has any power in the very sparse regime ($\beta > 3/4$).  This seems clear from the experiments reported in \figref{numeric2}.
This begs the question of whether there are any rank tests that have any (asymptotic) power in the very sparse regime.
We do not know the answer to that question, but are willing to conjecture that there are no such tests.

\paragraph{The two-sided problem}
We focused on the one-sided setting \eqref{model1}, effectively testing $\rho = 0$ versus $\rho > 0$. Knowing the sign of $\rho$ is not crucial, as one can apply a one-sided test for $\rho > 0$ to the transformed data $(X_1, -Y_1), \dots, (X_n, -Y_n)$. 
Less trivial is the case where there are three components
\[
(X, Y) \sim (1-\eps)\, \cN(0, \I) + \frac\eps2\, \cN(0, \Sigma_\rho) + \frac\eps2\, \cN(0, \Sigma_{-\rho}).
\]
We did not look at this model, in part because we wanted to test against a monotonic association (in the contamination component), which is perhaps the most popular alternative in a nonparametric context.

\subsection*{Acknowledgments}
We are grateful to two anonymous referees for asking pertinent questions that helped improved the paper, and for asking about a technical point in the proof of \thmref{HCrank} lacking rigor.

\small
\bibliographystyle{abbrvnat}
\bibliography{ref}

\begin{thebibliography}{23}
\providecommand{\natexlab}[1]{#1}
\providecommand{\url}[1]{\texttt{#1}}
\expandafter\ifx\csname urlstyle\endcsname\relax
  \providecommand{\doi}[1]{doi: #1}\else
  \providecommand{\doi}{doi: \begingroup \urlstyle{rm}\Url}\fi

\bibitem[Arias-Castro and Huang(2018)]{AriasCastro:2018wr}
E.~Arias-Castro and R.~Huang.
\newblock The sparse variance contamination model.
\newblock \emph{arXiv preprint arXiv:1807.10785}, 2018.

\bibitem[Arias-Castro and Wang(2016)]{AriasCastro:2016is}
E.~Arias-Castro and M.~Wang.
\newblock {Distribution-free tests for sparse heterogeneous mixtures}.
\newblock \emph{TEST}, 26\penalty0 (1):\penalty0 71--94, 2016.

\bibitem[Arratia et~al.(1990)Arratia, Goldstein, and
  Gordon]{arratia1990poisson}
R.~Arratia, L.~Goldstein, and L.~Gordon.
\newblock Poisson approximation and the {C}hen--{S}tein method.
\newblock \emph{Statistical Science}, pages 403--424, 1990.

\bibitem[Bilgrau et~al.(2016)Bilgrau, Eriksen, Rasmussen, Johnsen, Dybk{\ae}r,
  and B{\o}gsted]{bilgrau2016gmcm}
A.~E. Bilgrau, P.~S. Eriksen, J.~G. Rasmussen, H.~E. Johnsen, K.~Dybk{\ae}r,
  and M.~B{\o}gsted.
\newblock {GMCM}: Unsupervised clustering and meta-analysis using {G}aussian
  mixture copula models.
\newblock \emph{Journal of Statistical Software}, 70\penalty0 (2):\penalty0
  1--23, 2016.

\bibitem[Bogomolov and Heller(2018)]{Bogomolov:2018gt}
M.~Bogomolov and R.~Heller.
\newblock {Assessing replicability of findings across two studies of multiple
  features}.
\newblock \emph{Biometrika}, 105\penalty0 (3):\penalty0 505--516, 2018.

\bibitem[Cai and Wu(2014)]{cai2014optimal}
T.~T. Cai and Y.~Wu.
\newblock Optimal detection of sparse mixtures against a given null
  distribution.
\newblock \emph{IEEE Transactions on Information Theory}, 60\penalty0
  (4):\penalty0 2217--2232, 2014.

\bibitem[Cai et~al.(2011)Cai, Jeng, and Jin]{Cai:2011cb}
T.~T. Cai, X.~J. Jeng, and J.~Jin.
\newblock {Optimal detection of heterogeneous and heteroscedastic mixtures}.
\newblock \emph{Journal of the Royal Statistical Society: Series B (Statistical
  Methodology)}, 73\penalty0 (5):\penalty0 629--662, 2011.

\bibitem[Chatterjee(2007)]{chatterjee2007stein}
S.~Chatterjee.
\newblock Stein’s method for concentration inequalities.
\newblock \emph{Probability Theory and Related Fields}, 138\penalty0
  (1):\penalty0 305--321, 2007.

\bibitem[DasGupta(2008)]{dasgupta2008asymptotic}
A.~DasGupta.
\newblock \emph{Asymptotic Theory of Statistics and Probability}.
\newblock Springer, 2008.

\bibitem[Donoho and Jin(2004)]{Jin:2004fj}
D.~Donoho and J.~Jin.
\newblock {Higher criticism for detecting sparse heterogeneous mixtures}.
\newblock \emph{The Annals of Statistics}, 32\penalty0 (3):\penalty0 962--994,
  2004.

\bibitem[Donoho and Jin(2008)]{donoho2008higher}
D.~Donoho and J.~Jin.
\newblock Higher criticism thresholding: Optimal feature selection when useful
  features are rare and weak.
\newblock \emph{Proceedings of the National Academy of Sciences}, 105\penalty0
  (39):\penalty0 14790--14795, 2008.

\bibitem[Donoho and Jin(2015)]{donoho2015higher}
D.~Donoho and J.~Jin.
\newblock Higher criticism for large-scale inference, especially for rare and
  weak effects.
\newblock \emph{Statistical Science}, 30\penalty0 (1):\penalty0 1--25, 2015.

\bibitem[Gibbons and Chakraborti(2003)]{gibbons2003nonparametric}
J.~D. Gibbons and S.~Chakraborti.
\newblock \emph{Nonparametric Statistical Inference}, volume 168.
\newblock Marcel Dekker, Inc., 4th edition, 2003.

\bibitem[Hall and Jin(2010)]{hall2010innovated}
P.~Hall and J.~Jin.
\newblock Innovated higher criticism for detecting sparse signals in correlated
  noise.
\newblock \emph{The Annals of Statistics}, 38\penalty0 (3):\penalty0
  1686--1732, 2010.

\bibitem[Hoeffding(1951)]{hoeffding1951combinatorial}
W.~Hoeffding.
\newblock A combinatorial central limit theorem.
\newblock \emph{The Annals of Mathematical Statistics}, 22\penalty0
  (4):\penalty0 558--566, 1951.

\bibitem[Ingster(1997)]{ingster1997some}
Y.~I. Ingster.
\newblock Some problems of hypothesis testing leading to infinitely divisible
  distributions.
\newblock \emph{Mathematical Methods of Statistics}, 6\penalty0 (1):\penalty0
  47--69, 1997.

\bibitem[Lehmann and Romano(2005)]{lehmann2006testing}
E.~L. Lehmann and J.~P. Romano.
\newblock \emph{Testing Statistical Hypotheses}.
\newblock Springer, 3rd edition, 2005.

\bibitem[Li et~al.(2011)Li, Brown, Huang, and Bickel]{Li:2011gg}
Q.~Li, J.~B. Brown, H.~Huang, and P.~J. Bickel.
\newblock {Measuring reproducibility of high-throughput experiments}.
\newblock \emph{The Annals of Applied Statistics}, 5\penalty0 (3):\penalty0
  1752--1779, 2011.

\bibitem[Moscovich et~al.(2016)Moscovich, Nadler, and
  Spiegelman]{moscovich2016exact}
A.~Moscovich, B.~Nadler, and C.~Spiegelman.
\newblock On the exact {Berk--Jones} statistics and their p-value calculation.
\newblock \emph{Electronic Journal of Statistics}, 10\penalty0 (2):\penalty0
  2329--2354, 2016.

\bibitem[Redekop and Mladsi(2013)]{redekop2013faces}
W.~K. Redekop and D.~Mladsi.
\newblock The faces of personalized medicine: A framework for understanding its
  meaning and scope.
\newblock \emph{Value in Health}, 16\penalty0 (6):\penalty0 S4--S9, 2013.

\bibitem[Xu et~al.(2013)Xu, Hou, Hung, and Zou]{Xu:2013ip}
W.~Xu, Y.~Hou, Y.~S. Hung, and Y.~Zou.
\newblock {A comparative analysis of Spearman's rho and Kendall's tau in normal
  and contaminated normal models}.
\newblock \emph{Signal Processing}, 93\penalty0 (1):\penalty0 261--276, 2013.

\bibitem[Zhao(2015)]{zhao2015false}
S.~D. Zhao.
\newblock False discovery rate control for identifying simultaneous signals.
\newblock \emph{arXiv preprint arXiv:1512.04499}, 2015.

\bibitem[Zhao et~al.(2017)Zhao, Cai, and Li]{Zhao:9999ku}
S.~D. Zhao, T.~T. Cai, and H.~Li.
\newblock {Optimal detection of weak positive latent dependence between two
  sequences of multiple tests}.
\newblock \emph{Journal of Multivariate Analysis}, 160:\penalty0 169--184,
  2017.

\end{thebibliography}

\end{document}